\documentclass[12pt,lot,lof]{wcthesis}
\newcommand{\printmode}{}  

\title{On The Prime Numbers In Intervals}

\submitted{2015}  
\copyrightyear{2015}  
\author{Kyle D. Balliet}
\adviser{Professor Lin Tan}  
\department{Mathematics}

\usepackage{amssymb,latexsym,amsmath,epsfig,amsthm,enumitem,mathtools,xfrac,mdframed}

\usepackage{graphicx}

\usepackage{verbatim}

\usepackage{multirow}
\usepackage{longtable}

\usepackage{booktabs}

\ifdefined\printmode

\usepackage{url}

\else

\ifdefined\proquestmode

\usepackage{hyperref}
\hypersetup{bookmarksnumbered}

\makeatletter
\hypersetup{pdftitle=\@title,pdfauthor=\@author}
\makeatother

\else

\usepackage{hyperref}
\hypersetup{colorlinks,bookmarksnumbered}

\makeatletter
\hypersetup{pdftitle=\@title,pdfauthor=\@author}
\makeatother

\fi 
\fi 

\ifodd 0
\renewcommand{\maketitlepage}{}

\renewcommand*{\makeabstract}{}

\else

\abstract{Bertrand's postulate establishes that for all positive integers $n>1$ there exists a prime number between $n$ and $2n$.  We consider a generalization of this theorem as:  for integers $n\geq k\geq 2$ is there a prime number between $kn$ and $(k+1)n$?  We use elementary methods of binomial coefficients and the Chebyshev functions to establish the cases for $2\leq k\leq 8$.  We then move to an analytic number theory approach to show that there is a prime number in the interval $(kn, (k+1)n)$ for at least $n\geq k$ and $2\leq k\leq 519$.

We then consider Legendre's conjecture on the existence of a prime number between $n^2$ and $(n+1)^2$ for all integers $n\geq 1$.  To this end, we show that there is always a prime number between $n^2$ and $(n+1)^{2.000001}$ for all $n\geq 1$.  Furthermore, we note that there exists a prime number in the interval $[n^2,(n+1)^{2+\varepsilon}]$ for any $\varepsilon>0$ and $n$ sufficiently large.

We also consider the question of how many prime numbers there are between $n$ and $kn$ for positive integers $k$ and $n$ for each of our results and in the general case.  Furthermore, we show that the number of prime numbers in the interval $(n,kn)$ is increasing and that there are at least $k-1$ prime numbers in $(n,kn)$ for $n\geq k\geq 2$.

Finally, we compare our results to the prime number theorem and obtain explicit lower bounds for the number of prime numbers in each of our results.}

\approval{\noindent ``On The Prime Numbers In Intervals,'' a thesis prepared by Kyle D. Balliet in partial fulfillment of the requirements for the degree, Master of Arts, has been approved and accepted by the following:

\vspace{0.5in}
\begin{flushleft}
\hrulefill
\newline
Lin Tan, PhD.\\
Chairperson, Professor of Mathematics
\vspace{0.5in}

\hrulefill
\newline
James Mc Laughlin, PhD.\\
Associate Professor of Mathematics
\vspace{0.5in}

\hrulefill
\newline
Scott Parsell, PhD.\\
Associate Professor of Mathematics
\vspace{0.5in}

\hrulefill
\newline
R. Lorraine Bernotsky, D. Phil.\\
Associate Provost and Graduate Dean
\end{flushleft}}

\acknowledgements{
I would like to express my appreciation to everyone who allowed this thesis to come to fruition.  This includes not only the professors of West Chester University, but also to anyone who dared to question.

Thank you to Lin Tan for his encouragement, questions, and comments which helped shape this research.

I am also grateful to the staff of Bloomsburg University.  Perhaps most notably William Calhoun whose knowledge of number theory and correspondence would later inspire the work within these pages.

Most importantly for last.  Charlene who would allow me to drone on and on, and on some more, about Mathematics.  Sorry.}

\dedication{{\it Mia menso restis malfermitaj, nenio alia.}}

\listofnotation{\begin{tabular}{p{2cm} p{12cm}}
$\mathbb{N}$ & $\text{The set of positive integers:  }\{1, 2, 3, \ldots\}$\\
$\mathbb{Z}$ & $\text{The set of all integers:  }\{\ldots,-2, -1, 0, 1, 2, \ldots\}$\\
$\log x$ & $\text{The logarithm of } x \text{ with base } e \text{ (that is, the natural logarithm)}$\\
$\vartheta(x)$ & $\text{The first Chebyshev function:  }\displaystyle\sum_{p\leq x}\log p$\\
$\psi(x)$ & $\text{The second Chebyshev function:  } \displaystyle\sum_{\substack{p^k\leq x\\k\in\mathbb{N}}}\log p=\displaystyle\sum_{m=1}^{\infty}\vartheta(x^{1/m})$\\
$n!$ & The factorial of n:  $n!=1\cdot 2\cdot 3\cdots (n-1)\cdot n$\\
$\pi(n)$ & $\text{The number of prime numbers less than or equal to } n$\\
$\exp(x)$ & $\text{Equivalent notation for } e^x$\\
$A\setminus B$ & $\text{The relative complement of } B\text{ in } A.\text{  That is, }$\\
& $A\setminus B=\{x\in A\mid x\not\in B\}$\\
$u(n)$ & An upper bound approximation for $n!$:\\
& $n!<u(n)=\sqrt{2\pi}\,n^{n+\frac{1}{2}}\,e^{-n+\frac{1}{12n}}$\\
$l(n)$ & $\text{A lower bound approximation for }n!\text{: }$\\
& $n!>l(n)=\sqrt{2\pi}\,n^{n+\frac{1}{2}}\,e^{-n+\frac{1}{12n+1}}$\\
$[x]$ & The floor of $x$.  That is, $[x]=\max\{m\in\mathbb{Z}\mid m\leq x\}$\\
$\{x\}$ & The sawtooth function of $x$.  That is, $\{x\}=x-[x]$\\
${x\brace y}$ & Binomial coefficient with floored terms.  That is, ${x\brace y}=\delta(y,x)\binom{[x]}{[y]}$,\\
& where $\delta(y,x)=1$ if $\{x\}\geq \{y\}$ and $\delta(y,x)=[x-y]+1$ if $\{x\}<\{y\}$\\
$B_k(n,m)$ & $\text{The division of binomial coefficients:  } \left.\left\{\substack{\frac{(k+1)n}{m}\\\frac{kn}{m}}\right\} \middle/ \left\{\substack{\frac{(k+1)n}{2m}\\\frac{kn}{2m}}\right\}\right.$
\end{tabular}}

\begin{document}
\makefrontmatter

\newtheorem{Theorem}{\noindent Theorem}[section]
\newtheorem{Corollary}[Theorem]{\noindent Corollary}
\newtheorem{Lemma}[Theorem]{\noindent Lemma}
\newtheorem{Question}[Theorem]{\noindent Question}
\newtheorem{Example}[Theorem]{\noindent Example}
\newtheorem{Conjecture}[Theorem]{\noindent Conjecture}


\chapter{Introduction\label{intro}}

Bertrand's postulate is stated as:  For $n>1$ there is a prime number between $n$ and $2n$.  This, relatively simply stated, conjecture was stated by Joseph Bertrand in 1845 and was ultimately solved by Pafnuty Chebyshev in 1850, see \cite{chebyshev1850}.  Later, in 1919, it was shown by Srinivasa Ramanujan in \cite{ramanujan1919} using properties of the gamma function.  Finally, Paul Erd\H{o}s \cite{erdos1932,erdos2003} in 1932 showed the theorem using elementary properties, and approximations of binomial coefficients and the Chebyshev functions:
\begin{align*}
\vartheta(x)&=\displaystyle\sum_{p\leq x}\log p,\\
\psi(x)&=\displaystyle\sum_{\substack{p^k\leq x \\ k\in\mathbb{N}}}\log p.
\end{align*}

In 2006, Mohamed El Bachraoui \cite{bachraoui2006} showed that for $n>1$ there is a prime number between $2n$ and $3n$.  This similar result to Bertrand's postulate provided a smaller interval for a prime number to exist.  As an example, Bertrand's postulate guarantees $p\in(10,20)$ whereas Bachraoui guarantees $p\in(10,15)$.  Moreover, Bachraoui also questioned if there was a prime number between $kn$ and $(k+1)n$ for all $n\geq k\geq 2$.

Finally, in 2011 Andy Loo \cite{loo2011} shortened these intervals to $p\in(3n,4n)$ for $n>1$.  A. Loo went on to prove that as $n$ approaches infinity the number of prime numbers between $3n$ and $4n$ also tends to infinity - a result which is implied by the prime number theorem.

A common feature exists in each of these elementary proofs, mainly that each improvement requires checking more cases by hand/computer than the previous result.  Unfortunately, this is a condition of the method of proof used and the nature of the problem, as we shall see in Chapter 2.

The primary idea with these proofs is to consider the binomial coefficient $\binom{(k+1)n}{kn}$ and subdivide the `categories' of primes as small, medium, and large as:
\[T_1=\prod_{p\leq\sqrt{(k+1)n}}\mkern-18mu p^{\beta(p)},\quad T_2=\prod_{\sqrt{(k+1)n}<p\leq kn}\mkern-30mu p^{\beta(p)},\quad\text{and } T_3=\prod_{kn+1\leq p\leq (k+1)n}\mkern-30mu p\]
such that
\[\binom{(k+1)n}{kn}=T_1 T_2 T_3,\]
where $\beta(p)$ is the largest power of $p$, say $\alpha$, such that $p^{\alpha}$ divides $\binom{(k+1)n}{kn}$.  It should also be noted that we may readily calculate $\beta(p)$ by
\[\beta(p)=\sum_{t=1}^{\infty}\left(\left[\frac{(k+1)n}{p^t}\right]-\left[\frac{kn}{p^t}\right]-\left[\frac{n}{p^t}\right]\right),\]
a fact shown by P. Erd\H{o}s in \cite[p.~24]{erdos2003}.

We then bound $T_1$ and $T_2$ from above by easily computable approximations.  Similarly, we bound $\binom{(k+1)n}{kn}$ from below and show that
\[T_3=\binom{(k+1)n}{kn}\frac{1}{T_1 T_2}>1.\]
Since then $T_3>1$, there is at least one prime number between $kn$ and $(k+1)n$.

We will use methods similar to these in Section 2.1 to show that there is a prime number between $4n$ and $5n$ for all integers $n>2$.  However, in Section 2.2 we will not only apply these methods, but we will need to define a division of binomial coefficients in order to efficiently bound $T_2$.  We will define $B_k(n,m)$ as
\[B_k(n,m)=\left.{\frac{(k+1)n}{m}\brace \frac{kn}{m}} \middle/ {\frac{(k+1)n}{2m} \brace \frac{kn}{2m}}\right.\]
and establish both upper and lower bounds for this division.  In doing so we will be able to extend our proof technique to the case when $k=8$.

We then move from our elementary approach in favor of an analytic number theory approach that was utilized by S. Ramanujan in \cite{ramanujan1919} and J. Nagura in \cite{nagura1952}, showing that
\[\vartheta\left(\frac{(k+1)n}{k}\right)-\vartheta(n)>0\]
provided that $n$ and $k$ are integers such that $1\leq k\leq 519$ and $n\geq k$.  We will then show that this theorem yields a corollary that there is at least one prime number between $519n$ and $520n$ for all $n>15$.

We conclude Chapter 2 by showing that there exists a prime number $p$ such that $n^2<p<{(n+1)}^{2.000001}$.  This result is close to Legendre's conjecture, which asserts that there exists a prime number between $n^2$ and ${(n+1)}^2$.  Moreover, we note that there exists a prime number $p$ such that $n^2<p<(n+1)^{2+\varepsilon}$ for any $\varepsilon>0$ and some sufficiently large $n$.

In Chapter 3 we will turn our attention to determining the number of prime numbers between $n$ and $kn$ based upon our previous results.  We will also show that there are at least $k-1$ prime numbers between $n$ and $kn$ for any $k\geq 2$ and $n\geq k$.  In doing so, we will also note that the number of prime numbers between $n$ and $kn$ tends to infinity by showing that
\[\vartheta(kn)-\vartheta(n)>n\left(k-1-3.965\left(\frac{k}{\log^2(kn)}+\frac{1}{\log^2 n}\right)\right)>0,\]
where $\log^2 x=(\log x)^2$.

Lastly, we show that for every positive integer $m$ there exists a positive integer $L$ such that for all $n\geq L$ there are at least $m$ prime numbers in each of our intervals.  We then compare this to the prime number theorem which states that $\pi(x)\sim\frac{x}{\log x}$.
\chapter{Prime Numbers In Intervals\label{chapter2}}

In this chapter we will show that there is always a prime number in certain intervals for some positive integer $n$.

The elementary means employed in Sections 2.1 and 2.2 are essentially the same as those used by P. Erd\H{o}s \cite{erdos1932}, M. El Bachraoui \cite{bachraoui2006}, and A. Loo \cite{loo2011}.  The analytic means employed in Section 2.3 are the methods used by S. Ramanujan in \cite{ramanujan1919} and $\nolinebreak{\text{J. Nagura}}$ in \cite{nagura1952}.

We will extensively utilize the following bounds of the factorial (gamma) function as provided by H. Robbins \cite{robbins1955} in the next two sections and so we present them as a lemma.

\begin{Lemma}
Let $l(n)=\sqrt{2\pi}\,n^{n+\frac{1}{2}}\,e^{-n+\frac{1}{12n+1}}$ and $u(n)=\sqrt{2\pi}\,n^{n+\frac{1}{2}}\,e^{-n+\frac{1}{12n}}$, then $l(n)<n!<u(n)$ for $n\geq 1$.
\end{Lemma}

Moreover, we will also use the fact that $1\leq\delta(r,s)\leq s$ as shown in the following lemma.

\begin{Lemma}
Let $r$ and $s$ be real numbers satisfying $s>r\geq 1$ and ${s \brace r}=\delta(r,s)\binom{[s]}{[r]}$, then $1\leq \delta(r,s)\leq s$.
\end{Lemma}
\begin{proof}
Let $z>0$ and let $[z]$ be the greatest integer less than or equal to $z$.  Define $\{z\}=z-[z]$.  Let $r$ and $s$ be real numbers satisfying $s>r\geq1$.  Observe that the number of integers in the interval $(s-r,s]$ is $[s]-[s-r]$, which is $[r]$ if $\{s\}\geq\{r\}$ and $[r]+1$ if $\{s\}<\{r\}$.  Let $\mathbb{N}$ be the set of all natural numbers and define
\[{s \brace r}=\frac{\displaystyle\prod_{k\in(s-r,s]\cap\mathbb{N}}\mkern-24mu k}{\displaystyle\prod_{k\in(0,r]\cap\mathbb{N}}\mkern-16mu k}=\delta(r,s)\binom{[s]}{[r]}\]
\noindent where $\delta(r,s)=1$ if $\{s\}\geq\{r\}$ and $\delta(r,s)=[s-r]+1$ if $\{s\}<\{r\}$.  In either case, $1\leq \delta(r,s)\leq s$.
\end{proof}

\begin{Lemma}
The following are true:
\begin{enumerate}[font=\normalfont]
\item Let $c\geq\frac{1}{12}$ be a fixed constant.  For $x\geq\frac{1}{2}$, $\frac{u(x+c)}{l(c)l(x)}$ is increasing.
\item Let $c$ be a fixed positive constant and define $h_2(x)=\frac{u(c)}{l(x)l(c-x)}$.  Then $h'_2(x)>0$ when $\frac{1}{2}\leq x\leq \frac{c}{2}$, $h'_2(x)=0$ when $x=\frac{c}{2}$, and $h'_2(x)<0$ when $\frac{c}{2}<x\leq c-\frac{1}{2}$.
\end{enumerate}
\end{Lemma}
\begin{proof}
See Lemmas 1.4 and 1.5 of \cite{loo2011}.
\end{proof}

\section{Primes in the Interval [4$n$, 5$n$]}
\label{sec:Primes4n5n}

In this section we will show that there is always a prime number in the interval $(4n, 5n)$ for any positive integer $n>2$.  In order to do so, we will begin by showing two inequalities which will be vital to the main proof which follows.

For our main proof we will consider the binomial coefficient $\binom{5n}{4n}$ and subdivide the prime numbers composing $\binom{5n}{4n}$ into three products based on their size.  That is, we will categorize the prime numbers as
\[T_{1}=\prod_{p\leq \sqrt{5n}}\mkern-6mu p^{\beta(p)},\quad T_{2}=\prod_{\sqrt{5n} < p\leq 4n}\mkern-18mu p^{\beta(p)},\quad\text{and } T_{3}=\prod_{4n+1\leq p\leq 5n}\mkern-20mu p\]
so that
\[\binom{5n}{4n}=T_1 T_2 T_3.\]

We will then show that $T_3=\binom{5n}{4n}\frac{1}{T_1 T_2}$ is greater than 1 and therefore there is at least one prime number in $T_3$.  That is, there is at least one prime number in the interval $(4n, 5n)$.

\begin{Lemma}
The following inequalities hold:
\begin{enumerate}[font=\normalfont]
\item For $n\geq6818$, $e^{\sfrac{1}{(60n+1)}-\sfrac{1}{48n}-\sfrac{1}{12n}}\geq0.999986$.
\item For $n\geq1$, $e^{\sfrac{1}{30n}-\sfrac{1}{(24n+1)}-\sfrac{1}{(6n+1)}}\leq1$.
\item For $n\geq1$, $e^{\sfrac{1}{20n}-\sfrac{1}{(16n+1)}-\sfrac{1}{(4n+1)}}\leq1$.
\item For $n\geq6815$, $\frac{4n+3}{n-3}<4.002202$.
\end{enumerate}
\end{Lemma}
\textit{Proof} (1).  The following inequalities are equivalent for $n\geq6818$:
\begin{align*}
e^{\sfrac{1}{(60n+1)}-\sfrac{1}{48n}-\sfrac{1}{12n}}&\geq0.999986\\
\tfrac{1}{60n+1}-\tfrac{1}{48n}-\tfrac{1}{12n}&\geq\log 0.999986\\
\tfrac{252n+5}{2880n^{2}+48n}&\leq-\log 0.999986\approx0.000014.
\end{align*}
Now the left-hand side is decreasing in $n$ and so it suffices to verify the case when $n=6818$.  When $n=6818$, we obtain
\[\frac{1718141}{133877484384}<0.000013<-\log 0.999986.\]

\noindent\textit{Proof} (2).  The following inequalities are equivalent for $n\geq1$:
\begin{align*}
e^{\sfrac{1}{30n}-\sfrac{1}{(24n+1)}-\sfrac{1}{(6n+1)}}&\leq1\\
\tfrac{1}{30n}-\tfrac{1}{24n+1}-\tfrac{1}{6n+1}&\leq0\\
1&\leq756n^{2}+30n.
\end{align*}
Which clearly holds for all $n\geq1$.

\noindent\textit{Proof} (3).  The following inequalities are equivalent for $n\geq1$:
\begin{align*}
e^{\sfrac{1}{20n}-\sfrac{1}{(16n+1)}-\sfrac{1}{(4n+1)}}&\leq1\\
\tfrac{1}{20n}-\tfrac{1}{16n+1}-\tfrac{1}{4n+1}&\leq0\\
1&\leq336n^{2}+20n.
\end{align*}
Which clearly holds for all $n\geq1$.

\noindent\textit{Proof} (4).  The inequality follows directly by noting that $\frac{4n+3}{n-3}<4.002202$ is equivalent to $0<0.002202n-15.006606$ which clearly holds for all $n\geq6815$. $\hfill\qed$

\begin{Lemma}
For all $n\geq 6818$ the following inequality holds:
\[\frac{0.054886}{2^{\frac{n}{2}}n^{\frac{3}{2}}}{\left(\frac{3125}{256}\right)}^{\frac{n}{6}}>{(5n)}^{\frac{2.51012\sqrt{5n}}{\log(5n)}}.\]
\end{Lemma}
\begin{proof}The following are equivalent for $n\geq6818$:
\[\frac{0.054886}{2^{\frac{n}{2}}n^{\frac{3}{2}}}{\left(\frac{3125}{256}\right)}^{\frac{n}{6}}>{(5n)}^{\frac{2.51012\sqrt{5n}}{\log(5n)}}\]
\[\log0.054886-\frac{n}{2}\log2-\frac{3}{2}\log n+\frac{n}{6}\log3125-\frac{n}{6}\log256>2.51012\sqrt{5n}\]
\[\frac{n}{6}{\left[\log3125-3\log2-\log256\right]}>2.51012\sqrt{5n}+\frac{3}{2}\log n-\log0.054886\]
\[\frac{1}{6}\left[5\log5-11\log2\right]>\frac{2.51012\sqrt{5}}{\sqrt{n}}+\frac{3}{2}\left(\frac{\log n}{n}\right)-\frac{\log 0.054886}{n}.\]

Now the right-hand side is decreasing in $n$ and so it suffices to verify the case when $n=6818$.  When $n=6818$, we obtain
\[\frac{1}{6}\left[5\log5-11\log2\right]>0.0742>\frac{2.51012\sqrt{5}}{\sqrt{6818}}+\frac{3}{2}\left(\frac{\log6818}{6818}\right)-\frac{\log0.054886}{6818}.\qedhere\]
\end{proof}

We now proceed with the proof of our main theorem for this section; that is, there is always a prime number between $4n$ and $5n$ for all integers $n>2$.

\vspace{10pt}
\begin{Theorem}
For any positive integer $n>2$ there is a prime number between $4n$ and $5n$.
\end{Theorem}
\begin{proof}It can be easily verified that for $n=3, 4, \ldots, 6817$ there is always a prime number between $4n$ and $5n$.  Now let $n\geq6818$ and consider:
\[\binom{5n}{4n}=\frac{(4n+1)(4n+2)\cdots(5n)}{1\cdot2\cdots n}.\]

The product of primes between $4n$ and $5n$, if there are any, divides $\binom{5n}{4n}$.

Following the notation used in \cite{bachraoui2006,erdos1932,loo2011}, we let
\[T_{1}=\prod_{p\leq \sqrt{5n}}\mkern-6mu p^{\beta(p)},\quad T_{2}=\prod_{\sqrt{5n} < p\leq 4n}\mkern-18mu p^{\beta(p)},\quad \text{ and } T_{3}=\prod_{4n+1\leq p\leq 5n}\mkern-20mu p\]
such that
\[\binom{5n}{4n}=T_{1}T_{2}T_{3}.\]

The prime decomposition of $\binom{5n}{4n}$ implies that the powers in $T_2$ are less than 2; see \cite[p.~24]{erdos2003} for the prime decomposition of $\binom{n}{j}$.  In addition, the prime decomposition of $\binom{5n}{4n}$, see \cite[p.~24]{erdos2003}, yields the upper bound for $T_{1}$:
\[T_{1}<{(5n)}^{\pi(\sqrt{5n})}.\]

But $\pi(x)\leq\frac{1.25506x}{\log x}$, see \cite{rosser1962}.  So we obtain

\[T_{1}<{(5n)}^{\pi(\sqrt{5n})}\leq{(5n)}^{\frac{2.51012\sqrt{5n}}{\log(5n)}}.\]

\noindent Now let $A={5n/2 \brace 2n}$ and $B={5n/3 \brace 4n/3}$ and observe the following for a prime number $p$ in $T_2$:

\vspace{12pt}
\noindent$\circ$ If $\frac{5n}{2}<p\leq 4n$, then
\[n<p\leq 4n<5n<2p.\]
\noindent Hence $\beta(p)=0$.

\vspace{12pt}
\noindent$\circ$ Clearly $\displaystyle\prod_{2n<p\leq \frac{5n}{2}}\mkern-12mu p$ divides $A$.

\newpage

\noindent$\circ$ If $\frac{5n}{3}<p\leq2n$, then
\[n<p<2p\leq4n<5n<3p.\]
\noindent Hence $\beta(p)=0$.

\vspace{16pt}
\noindent$\circ$ Clearly $\displaystyle\prod_{\frac{4n}{3}<p\leq \frac{5n}{3}}\mkern-12mu p$ divides $B$.

\vspace{16pt}
\noindent$\circ$ If $\frac{5n}{4}<p\leq\frac{4n}{3}$, then
\[n<p<3p\leq4n<5n<4p.\]
\noindent Hence $\beta(p)=0$.

\vspace{16pt}
\noindent$\circ$ If $n<p\leq\frac{5n}{4}$, then
\[\frac{n}{2}<p<2n<2p\leq\frac{5n}{2}<3p.\]
\noindent Hence $\displaystyle\prod_{n<p\leq \frac{5n}{4}}\mkern-8mu p$ divides $A$.

\vspace{16pt}
\noindent$\circ$ If $\frac{5n}{6}<p\leq n$, then
\[p\leq n<2p<4p\leq 4n<5p\leq 5n<6p.\]
\noindent Hence $\beta(p)=0$.

\vspace{16pt}
\noindent$\circ$ If $\frac{2n}{3}<p\leq\frac{5n}{6}$, then
\[\frac{n}{3}<p<\frac{4n}{3}<2p\leq\frac{5n}{3}<3p.\]
\noindent Hence $\displaystyle\prod_{\frac{2n}{3}<p\leq \frac{5n}{6}}\mkern-12mu p$ divides $B$.

\newpage

\noindent$\circ$ If $\frac{5n}{8}<p\leq\frac{2n}{3}$, then
\[p<n<2p<6p\leq 4n<7p<5n<8p.\]
\noindent Hence $\beta(p)=0$.

\noindent$\circ$ If $\frac{n}{2}<p\leq\frac{5n}{8}$, then
\[\frac{n}{2}<p<3p<2n<4p\leq\frac{5n}{2}<5p.\]
\noindent Hence $\displaystyle\prod_{\frac{n}{2}<p\leq \frac{5n}{8}}\mkern-8mu p$ divides $A$.

\noindent$\circ$ If $\frac{5n}{11}<p\leq\frac{n}{2}$, then
\[p<2p\leq n<3p<8p\leq 4n<9p<10p\leq 5n<11p.\]
\noindent Hence $\beta(p)=0$.

\noindent$\circ$ If $\frac{4n}{9}<p\leq\frac{5n}{11}$, then
\[\frac{n}{3}<p<2p<\frac{4n}{3}<3p<\frac{5n}{3}<4p.\]
\noindent Hence $\displaystyle\prod_{\frac{4n}{9}<p\leq \frac{5n}{11}}\mkern-12mu p$ divides $B$.

\noindent$\circ$ If $\frac{5n}{12}<p\leq\frac{4n}{9}$, then
\[p<2p<n<3p<9p\leq 4n<10p<11p<5n<12p.\]
\noindent Hence $\beta(p)=0$.

\noindent$\circ$ If $\frac{n}{3}<p\leq\frac{5n}{12}$, then
\[\frac{n}{3}<p<3p<\frac{4n}{3}<4p\leq\frac{5n}{3}<5p.\]
\noindent Hence $\displaystyle\prod_{\frac{n}{3}<p\leq \frac{5n}{12}}\mkern-8mu p$ divides $B$.

\noindent$\circ$ If $\frac{5n}{16}<p\leq\frac{n}{3}$, then
\[p<2p<3p\leq n<4p<12p\leq 4n<13p<14p<15p\leq 5n<16p.\]
\noindent Hence $\beta(p)=0$.

\noindent$\circ$ If $\frac{2n}{7}<p\leq\frac{5n}{16}$, then
\[p<\frac{n}{2}<2p<6p<2n<7p<8p\leq\frac{5n}{2}<9p.\]
\noindent Hence $\displaystyle\prod_{\frac{2n}{7}<p\leq \frac{5n}{16}}\mkern-12mu p$ divides $A$.

\noindent$\circ$ If $\frac{5n}{18}<p\leq\frac{2n}{7}$, then
\[p<2p<3p<n<4p<14p\leq 4n<15p<16p<17p<5n<18p.\]
\noindent Hence $\beta(p)=0$.

\noindent$\circ$ If $\frac{n}{4}<p\leq\frac{5n}{18}$, then
\[p<\frac{n}{2}<2p<7p<2n<8p<9p\leq\frac{5n}{2}<10p.\]
\noindent Hence $\displaystyle\prod_{\frac{n}{4}<p\leq \frac{5n}{18}}\mkern-8mu p$ divides $A$.

By the fact that $\displaystyle\prod_{p\leq x}p<4^x$, see \cite[p.~167]{erdos2003}, we obtain\\
\noindent $\circ$  $\displaystyle\prod_{\sqrt{5n}<p\leq \frac{n}{4}}\mkern-12mu p<4^{\frac{n}{4}}=2^{\frac{n}{2}}.$

Now, to summarize, we obtain
\[T_{2}=\displaystyle\prod_{\sqrt{5n} < p\leq 4n}\mkern-14mu p^{\beta(p)}<2^{\frac{n}{2}}AB.\]

By Lemmas 2.0.1 and 2.1.1, we obtain
\begin{align*}
\binom{5n}{4n}&=\frac{(5n)!}{(4n)!(n!)}\\
&>\frac{l(5n)}{u(4n)u(n)}\\
&=\sqrt{\frac{5}{8\pi n}}{\left(\frac{3125}{256}\right)}^{n}e^{\frac{1}{60n+1}-\frac{1}{48n}-\frac{1}{12n}}\\
&>0.446024n^{-\frac{1}{2}}{\left(\frac{3125}{256}\right)}^{n}.
\end{align*}

Similarly, by Lemmas 2.0.1, 2.0.2, 2.0.3, and 2.1.1, we obtain
\begin{align*}
A&={\frac{5n}{2} \brace 2n}=\delta(2n,5n/2)\binom{[\frac{5n}{2}]}{2n}\\
&\leq\frac{5n}{2}\binom{[\frac{5n}{2}]}{2n}\leq\frac{5n}{2}\cdot\frac{u([\frac{5n}{2}])}{l(2n)l([\frac{5n}{2}]-2n)}\\
&\leq\frac{5n}{2}\cdot\frac{u(\frac{5n}{2})}{l(2n)l(\frac{n}{2})}\\
&=\frac{5}{4}\sqrt{\frac{5n}{\pi}}{\left(\frac{3125}{256}\right)}^{\frac{n}{2}}e^{\frac{1}{30n}-\frac{1}{24n+1}-\frac{1}{6n+1}}\\
&<1.576958n^{\frac{1}{2}}{\left(\frac{3125}{256}\right)}^{\frac{n}{2}}.
\end{align*}

Lastly, we obtain the upper bound approximation for $B$ as:
\begin{align*}
B&={\frac{5n}{3} \brace \frac{4n}{3}}=\delta(4n/3,5n/3)\binom{[\frac{5n}{3}]}{[\frac{4n}{3}]}\\
&\leq\frac{5n}{3}\binom{[\frac{5n}{3}]}{[\frac{4n}{3}]}\\
&=\frac{5n}{3}\cdot\frac{[\frac{4n}{3}]+1}{[\frac{5n}{3}]-[\frac{4n}{3}]}\cdot\binom{[\frac{5n}{3}]}{[\frac{4n}{3}]+1}\\
&\leq\frac{5n}{3}\cdot\frac{\frac{4n}{3}+1}{\frac{5n}{3}-1-\frac{4n}{3}}\cdot\frac{u([\frac{5n}{3}])}{l([\frac{4n}{3}]+1)l([\frac{5n}{3}]-([\frac{4n}{3}]+1))}\\
&\leq\frac{5n}{3}\cdot\frac{4n+3}{n-3}\cdot\frac{u(\frac{5n}{3})}{l(\frac{4n}{3})l(\frac{n}{3})}\\
&=\sqrt{\frac{125}{24\pi}}{\left(\frac{4n+3}{n-3}\right)}n^{\frac{1}{2}}{{\left(\frac{3125}{256}\right)}^{\frac{n}{3}}}e^{\frac{1}{20n}-\frac{1}{16n+1}-\frac{1}{4n+1}}\\
&<5.153158n^{\frac{1}{2}}{{\left(\frac{3125}{256}\right)}^{\frac{n}{3}}}.
\end{align*}

Thus we obtain
\begin{align*}
{T_3}&=\binom{5n}{4n}\frac{1}{{T_2}{T_1}}\\
&>\binom{5n}{4n}\frac{1}{2^{\frac{n}{2}}AB}\cdot\frac{1}{{(5n)}^{\frac{2.51012\sqrt{5n}}{\log(5n)}}}\\
&>\frac{0.446024n^{-\frac{1}{2}}{\left(\frac{3125}{256}\right)}^{n}}{2^{\frac{n}{2}}\left(1.576958n^{\frac{1}{2}}{\left(\frac{3125}{256}\right)}^{\frac{n}{2}}\right)\left(5.153158n^{\frac{1}{2}}{{\left(\frac{3125}{256}\right)}^{\frac{n}{3}}}\right)}\cdot\frac{1}{{(5n)}^{\frac{2.51012\sqrt{5n}}{\log(5n)}}}\\
&>\frac{0.054886}{2^{\frac{n}{2}}n^{\frac{3}{2}}}{\left(\frac{3125}{256}\right)}^{\frac{n}{6}}\cdot\frac{1}{{(5n)}^{\frac{2.51012\sqrt{5n}}{\log(5n)}}}>1,
\end{align*}
where the last inequality follows by Lemma 2.1.2.  Consequently the product $T_3$ of prime numbers between $4n$ and $5n$ is greater than 1 and therefore the existence of such numbers is proven.
\end{proof}

With the proof of the previous theorem complete we may also show that there is always a prime number between $n$ and $(5n+15)/4$ for all positive integers $n>2$ as in the following theorem.

\begin{Theorem}
For any positive integer $n>2$ there exists a prime number $p$ satisfying $n<p<\frac{5(n+3)}{4}$.
\end{Theorem}
\begin{proof}
When $n=3$, we obtain $3<5<7<\frac{15}{2}$.  Let $n\geq 4$.  By the division algorithm $4\mid(n+r)$ for some $r\in\{0,1,2,3\}$ and by Theorem 2.1.3 there exists a prime number $p$ such that $p\in\left(n+r,\frac{5(n+r)}{4}\right)$.  Since $\left(n+r,\frac{5(n+r)}{4}\right)$ is contained in $\left(n,\frac{5(n+3)}{4}\right)$ for all $0\leq r\leq 3$ and $n>2$, $p\in\left(n,\frac{5(n+3)}{4}\right)$ as desired.
\end{proof}
\section{Primes in the Interval [8$n$, 9$n$]}
\label{sec:Primes8n9n}

If we attempt to continue using the process from the previous section, then we would run into an issue.  For example, consider trying to prove that there is always a prime number between $5n$ and $6n$ for $n>1$.  We consider the binomial $\binom{6n}{5n}$ and let
\[T_{1}=\prod_{p\leq \sqrt{6n}}\mkern-5mu p^{\beta(p)},\quad T_{2}=\prod_{\sqrt{6n} < p\leq 5n}\mkern-16mu p^{\beta(p)},\quad\text{and } T_{3}=\prod_{5n+1\leq p\leq 6n}\mkern-18mu p.\]

As in Section 2.1, we approximate $T_1$ easily enough as $T_1<{(6n)}^{\pi(\sqrt{6n})}$, however when we attempt to approximate $T_2$ the issue presents itself.

To see this, let $m$ be a natural number and consider a prime number $p$ satisfying $\frac{6n}{m+1}<p\leq\frac{5n}{m}$.  Now if, $m<5$, then $n<\frac{6n}{m+1}<p\leq mp\leq 5n<6n<(m+1)p$ and hence $p$ does not divide $\binom{6n}{5n}$.

Therefore,
\begin{align*}
T_2 &=\prod_{\sqrt{6n}<p\leq 5n}\mkern-14mu p^{\beta_p}\\
&=\prod_{\sqrt{6n}<p\leq n}\mkern-12mu p^{\beta_p}\prod_{n<p\leq\frac{6n}{5}}\mkern-8mu p^{\beta_p}\prod_{\frac{5n}{4}<p\leq\frac{3n}{2}}\mkern-10mu p^{\beta_p}\prod_{\frac{5n}{3}<p\leq 2n}\mkern-8mu p^{\beta_p}\prod_{\frac{5n}{2}<p\leq 3n}\mkern-8mu p^{\beta_p}.
\end{align*}

Certainly
\[\prod_{\frac{5n}{4}<p\leq\frac{3n}{2}}\mkern-12mu p\,\,\cdot\prod_{\frac{5n}{2}<p\leq 3n}\mkern-12mu p, \quad \prod_{\frac{5n}{3}<p\leq 2n}\mkern-8mu p,\quad \text{and } \prod_{n<p\leq\frac{6n}{5}}\mkern-8mu p\]
divide $\binom{3n}{5n/2}, \binom{2n}{5n/3}, \text{and }\binom{6n/5}{n}$, respectively.  Then we have
\[T_3=\binom{6n}{5n}\frac{1}{T_1 T_2}>\frac{\binom{6n}{5n}}{\binom{3n}{5n/2}\binom{2n}{5n/3}\binom{6n/5}{n}}\cdot\prod_{\sqrt{6n}<p\leq n}p^{-\beta(p)}\cdot\frac{1}{T_1}.\]

However,
\[1>\frac{\binom{6n}{5n}}{\binom{3n}{5n/2}\binom{2n}{5n/3}\binom{6n/5}{n}}\]
for all $n>40$ and thus the method is inconclusive.

Therefore we need to find a sharper approximation for each of the products of primes which compose $T_2$.  We will do so by defining $B_k(n,m)$ to be a division of binomial coefficients as
\[B_k(n,m)=\left.{\frac{(k+1)n}{m}\brace \frac{kn}{m}} \middle/ {\frac{(k+1)n}{2m} \brace \frac{kn}{2m}}\right..\]

We will then bound this approximation from both above and below.  The reason for acquiring a lower bound is due to the fact, as in our example, just because $\prod_{\frac{5n}{2}<p\leq 3n}p$ divides $B_5(n,2)$ does not mean that $B_5(n,2)$ is an upper bound for the product of primes.  Therefore we also need to show that
\[B_5(n,2)\cdot\prod_{\frac{5n}{2}<p\leq 3n}\frac{1}{p}\geq 1\]
and thus $B_5(n,2)$ is an upper bound for the product of primes between $\frac{5n}{2}$ and $3n$.

While our proof in this section is not entirely elementary due to the methods utilized in Lemmas 2.2.4 and 2.2.5, the results of this section provide insight and proof techniques which may be used to show other cases via elementary means.  Also, the proof of Lemma 2.2.4 provides a shift into the proof techniques of Sections 2.3 and 2.4.

Our first priority is to show that if $k,m$, and $n$ are natural numbers with $\nolinebreak{n\geq m\geq 2}$ and $k\geq2$, then any prime number satisfying $\frac{kn}{m}<p\leq\frac{(k+1)n}{m}$ divides $B_k(n,m)$ as in the next lemma.

\begin{Lemma}
For $k,m,n\in\mathbb{N}$ with $n\geq m\geq2$ and $k\geq2$.  If $p$ is a prime number such that $\nolinebreak{\frac{kn}{m}< p\leq\frac{(k+1)n}{m}}$, then $p$ divides

\[B_k(n,m)=\left.{\frac{(k+1)n}{m}\brace \frac{kn}{m}} \middle/ {\frac{(k+1)n}{2m} \brace \frac{kn}{2m}}\right..\]
\end{Lemma}
\begin{proof}
Let $k,m,n\in\mathbb{N}$ with $n\geq m\geq2$ and $k\geq2$.  Let $p$ be a prime number such that $\frac{kn}{m}<p\leq\frac{(k+1)n}{m}$ and observe that

\[\frac{kn}{2m}<\frac{(k+1)n}{2m}<\frac{kn}{m}<p\leq\frac{(k+1)n}{m}.\]

Hence $p$ divides $B_k(n,m)$ as desired.
\end{proof}

We now turn our attention to showing an upper approximation for $B_k(n,m)$.  To accomplish this we will utilize Lemma 2.0.1 to approximate the factorial function.

\begin{Lemma}
For all $k,m,n\in\mathbb{N}$ with $n\geq m\geq2$ and $k\geq2$,
\begin{align*}
B_k(n,m)&<\frac{n(k+1)(kn+m)(kn+n+2m)(n+2m)}{4\sqrt{2}m^3 (n-m)}{\left(\frac{{(k+1)}^{k+1}}{k^k}\right)}^{\frac{n}{2m}}e^E\\
&\leq\frac{n(k+1)(kn+m)(kn+n+2m)(n+2m)}{4\sqrt{2}m^3 (n-m)}{\left(\frac{{(k+1)}^{k+1}}{k^k}\right)}^{\frac{n}{2m}}e^{\frac{12029}{111150}}
\end{align*}
where
\[E=\frac{(2k^2+5k+2)m}{12k(k+1)n}-\frac{m}{12kn+m}-\frac{m}{12n+m}-\frac{m}{6(k+1)n+m}.\]
\end{Lemma}
\begin{proof}
By Lemmas 2.0.1, 2.0.2, and 2.0.3, we obtain
\begin{align*}
{\frac{(k+1)n}{2m}\brace\frac{kn}{2m}}&=\delta\left(kn/m,(k+1)n/m\right)
{\left[(k+1)n/m\right]\choose \left[kn/m\right]}\\
&\leq\frac{(k+1)n}{m}\cdot{\left[(k+1)n/m\right]\choose \left[kn/m\right]}\\
&=\frac{(k+1)n}{m}\cdot\frac{\left[\frac{kn}{m}\right]+1}{\left[\frac{(k+1)n}{m}\right]-\left[\frac{kn}{m}\right]}\binom{\left[(k+1)n/m\right]}{\left[kn/m\right]+1}\\
&\leq\frac{(k+1)n}{m}\cdot\frac{\frac{kn}{m}+1}{\frac{(k+1)n}{m}-\frac{kn}{m}-1}\cdot\frac{u\left(\frac{(k+1)n}{m}\right)}{l\left(\frac{kn}{m}\right)l\left(\frac{n}{m}\right)}\\
&=\sqrt{\frac{(k+1)n}{2\pi km}}\cdot\frac{(k+1)(kn+m)}{n-m}\cdot{\left(\frac{{(k+1)}^{k+1}}{k^k}\right)}^{\frac{n}{m}}e^{E_u},
\end{align*}
where
\[E_u=\frac{m}{12(k+1)n}-\frac{m}{12kn+m}-\frac{m}{12n+m}.\]

Similarly, we obtain the lower bound approximation for ${(k+1)n/2m\brace kn/2m}$ as:
\begin{align*}
{\frac{(k+1)n}{2m}\brace\frac{kn}{2m}}&=\delta\left(kn/2m,(k+1)n/2m\right)\binom{\left[(k+1)n/2m\right]}{\left[kn/2m\right]}\\
&\geq1\cdot\binom{\left[(k+1)n/2m\right]}{\left[kn/2m\right]}\\
&=\frac{1}{\left(\left[\frac{(k+1)n}{2m}\right]+1\right)\left(\left[\frac{(k+1)n}{2m}\right]-\left[\frac{kn}{2m}\right]\right)}\cdot\frac{\left(\left[\frac{(k+1)n}{2m}\right]+1\right)!}{\left[\frac{kn}{2m}\right]!\cdot\left(\left[\frac{(k+1)n}{2m}\right]-\left[\frac{kn}{2m}\right]-1\right)!}\\
&>\frac{1}{\left(\frac{(k+1)n}{2m}+1\right)\left(\frac{(k+1)n}{2m}-\frac{kn}{2m}+1\right)}\cdot\frac{l\left(\left[\frac{(k+1)n}{2m}\right]+1\right)}{u\left(\left[\frac{kn}{2m}\right]\right)u\left(\left[\frac{(k+1)n}{2m}\right]-\left[\frac{kn}{2m}\right]-1\right)}\\
&\geq\frac{4m^2}{((k+1)n+2m)(n+2m)}\cdot\frac{l\left(\frac{(k+1)n}{2m}\right)}{u\left(\frac{kn}{2m}\right)u\left(\frac{n}{2m}\right)}\\
&=\frac{4m^2}{(kn+n+2m)(n+2m)}\sqrt{\frac{(k+1)m}{\pi kn}}{\left(\frac{{(k+1)}^{k+1}}{k^k}\right)}^{\frac{n}{2m}}e^{E_l},
\end{align*}
where
\[E_l=\frac{m}{6(k+1)n+m}-\frac{m}{6kn}-\frac{m}{6n}=\frac{m}{6(k+1)n+m}-\frac{(k+1)m}{6kn}.\]

Therefore
\[B_k(n,m)<\frac{n(k+1)(kn+m)(kn+n+2m)(n+2m)}{4\sqrt{2}m^3 (n-m)}{\left(\frac{{(k+1)}^{k+1}}{k^k}\right)}^{\frac{n}{2m}}e^E,\]
where
\[E=E_u-E_l=\frac{m}{n}\left(\frac{2k^2+5k+2}{12k(k+1)}-\frac{1}{12k+m/n}-\frac{1}{12+m/n}-\frac{1}{6(k+1)+m/n}\right).\]

Since $k\geq2$, clearly $k^2+3k>1$.  In other words, $2k^2+5k+2>k^2+2k+3$.  Dividing both sides by $12k(k+1)$ and rewriting the right-hand side, we obtain
\[\frac{2k^2+5k+2}{12k(k+1)}>\frac{1}{12}+\frac{1}{12k}+\frac{1}{6(k+1)}.\]

Moreover, since $n\geq m\geq2$, $0<\frac{m}{n}\leq1$ and hence
\begin{align*}
\frac{2k^2+5k+2}{12k(k+1)}&>\frac{1}{12}+\frac{1}{12k}+\frac{1}{6(k+1)}\\
&>\frac{1}{12+m/n}+\frac{1}{12k+m/n}+\frac{1}{6(k+1)+m/n}\\
&=\frac{n}{12n+m}+\frac{n}{12kn+m}+\frac{n}{6(k+1)n+m}.
\end{align*}

Multiplying both sides by $\frac{m}{n}$, we obtain
\[\frac{(2k^2+5k+2)m}{12k(k+1)n}>\frac{m}{12n+m}+\frac{m}{12kn+m}+\frac{m}{6(k+1)n+m}.\]

Hence $E$ is decreasing in $n$.  Furthermore, since $n\geq m$, $n\in[m,\infty)$ and hence $E$ is maximum when $n=m$.  When $n=m$, we obtain
\begin{align*}
E&=\frac{2k^2+5k+2}{12k(k+1)}-\frac{1}{12k+1}-\frac{1}{13}-\frac{1}{6(k+1)+1}\\
&=\frac{1008k^4+2268k^3+2684k^2+1463k+182}{11232k^4+25272k^3+15132k^2+1092k}\\
&=\frac{1008+2268/k+2684/k^2+1463/k^3+182/k^4}{11232+25272/k+15132/k^2+1092/k^3}\\
&\geq\frac{7}{78}
\end{align*}

Now since $E$ is decreasing in $k$ and $k\geq2$ we have that $E$ is maximum when $k=2$.  Hence $\frac{7}{78}\leq E\leq\frac{12029}{111150}$.  Thus we obtain
\begin{align*}
B_k(n,m)&<\frac{n(k+1)(kn+m)(kn+n+2m)(n+2m)}{4\sqrt{2}m^3 (n-m)}{\left(\frac{{(k+1)}^{k+1}}{k^k}\right)}^{\frac{n}{2m}}e^E\\
&\leq\frac{n(k+1)(kn+m)(kn+n+2m)(n+2m)}{4\sqrt{2}m^3 (n-m)}{\left(\frac{{(k+1)}^{k+1}}{k^k}\right)}^{\frac{n}{2m}}e^{\frac{12029}{111150}}.\qedhere
\end{align*}
\end{proof}

We must show that $B_k(n,m)$ is an upper bound for the product of prime numbers between $\frac{kn}{m}$ and $\frac{(k+1)n}{m}$.  One way to accomplish this task is to bound $B_k(n,m)$ from below, bound the product of primes from above, and determine their inequality.  We will do so for $k=8$ and therefore show that $B_8(n,m)$ is an upper bound for the product of primes between $\frac{8n}{m}$ and $\frac{9n}{m}$.

\begin{Lemma}
For all $k,m,n\in\mathbb{N}$ with $n\geq m\geq2$ and $k\geq2$,
\begin{align*}
B_k(n,m)&>\frac{\sqrt{2}m^3(n-2m)}{n(k+1)(kn+n+m)(n+m)(kn+2m)}{\left(\frac{{(k+1)}^{k+1}}{k^k}\right)}^{\frac{n}{2m}}e^F\\
&\geq\frac{\sqrt{2}m^3(n-2m)}{n(k+1)(kn+n+m)(n+m)(kn+2m)}{\left(\frac{{(k+1)}^{k+1}}{k^k}\right)}^{\frac{n}{2m}}e^{\frac{5}{84}}
\end{align*}
where
\[F=\frac{m}{12(k+1)n+m}+\frac{m}{6kn+m}+\frac{m}{6n+m}-\frac{(k^2+4k+1)m}{12k(k+1)n}.\]
\end{Lemma}
\begin{proof}
Let
\[F_l=\frac{m}{12(k+1)n+m}-\frac{(k+1)m}{12kn}\]
and
\[F_u=\frac{m}{6(k+1)n}-\frac{m}{6kn+m}-\frac{m}{6n+m},\]
then by Lemmas 2.0.1, 2.0.2, and 2.0.3, we obtain
\vspace{-9pt}
\begin{align*}
{\frac{(k+1)n}{m}\brace\frac{kn}{m}}&=\delta(kn/m, (k+1)n/m)\binom{[(k+1)n/m]}{[kn/m]}\\
&\geq 1\cdot\binom{[(k+1)n/m]}{[kn/m]}\\
&=\frac{1}{\left(\left[\frac{(k+1)n}{m}\right]+1\right)\left(\left[\frac{(k+1)n}{m}\right]-\left[\frac{kn}{m}\right]\right)}\cdot\frac{\left(\left[\frac{(k+1)n}{m}\right]+1\right)!}{\left[\frac{kn}{m}\right]!\cdot\left(\left[\frac{(k+1)n}{m}\right]-\left[\frac{kn}{m}\right]-1\right)!}\\
&>\frac{1}{\left(\frac{(k+1)n}{m}+1\right)\left(\frac{(k+1)n}{m}-\frac{kn}{m}+1\right)}\cdot\frac{l\left(\left[\frac{(k+1)n}{m}\right]+1\right)}{u\left(\left[\frac{kn}{m}\right]\right)u\left(\left[\frac{(k+1)n}{m}\right]-\left[\frac{kn}{m}\right]-1\right)}\\
&\geq\frac{m^2}{(kn+n+m)(n+m)}\cdot\frac{l\left(\frac{(k+1)n}{m}\right)}{u\left(\frac{kn}{m}\right)u\left(\frac{n}{m}\right)}\\
&=\frac{m^2}{(kn+n+m)(n+m)}\sqrt{\frac{(k+1)m}{2\pi kn}}{\left(\frac{{(k+1)}^{k+1}}{k^k}\right)}^{\frac{n}{m}}e^{F_l}.
\end{align*}

Similarly, we obtain the upper bound approximation of ${(k+1)n/2m \brace kn/2m}$ as:
\vspace{-9pt}
\begin{align*}
{\frac{(k+1)n}{2m}\brace\frac{kn}{2m}}&=\delta(kn/2m,(k+1)n/2m)\binom{[(k+1)n/2m]}{[kn/2m]}\\
&\leq\frac{(k+1)n}{2m}\cdot\binom{[(k+1)n/2m]}{[kn/2m]}\\
&=\frac{(k+1)n}{2m}\cdot\frac{\left[\frac{kn}{2m}\right]+1}{\left[\frac{(k+1)n}{2m}\right]-\left[\frac{kn}{2m}\right]}\binom{[(k+1)n/2m]}{[kn/2m]+1}\\
&\leq\frac{(k+1)n}{2m}\cdot\frac{\frac{kn}{2m}+1}{\frac{(k+1)n}{2m}-1-\frac{kn}{2m}}\cdot\frac{u\left(\left[\frac{(k+1)n}{2m}\right]\right)}{l\left(\left[\frac{kn}{2m}\right]+1\right)l\left(\left[\frac{(k+1)n}{2m}\right]-\left[\frac{kn}{2m}\right]-1\right)}\\
&\leq\frac{(k+1)n}{2m}\cdot\frac{kn+2m}{n-2m}\cdot\frac{u\left(\frac{(k+1)n}{2m}\right)}{l\left(\frac{kn}{2m}\right)l\left(\frac{n}{2m}\right)}\\
&=\frac{n(k+1)(kn+2m)}{2m(n-2m)}\sqrt{\frac{(k+1)m}{\pi kn}}{\left(\frac{{(k+1)}^{k+1}}{k^k}\right)}^{\frac{n}{2m}}e^{F_u}.
\end{align*}

Thus
\[B_k(n,m)>\frac{\sqrt{2}m^3(n-2m)}{n(k+1)(kn+n+m)(n+m)(kn+2m)}{\left(\frac{{(k+1)}^{k+1}}{k^k}\right)}^{\frac{n}{2m}}e^F\]
where
\begin{align*}
F&=F_l-F_u\\
&=\frac{m}{12(k+1)n+m}+\frac{m}{6kn+m}+\frac{m}{6n+m}-\frac{(k^2+4k+1)m}{12k(k+1)n}\\
&=\frac{m}{n}\left(\frac{n}{12(k+1)+m/n}+\frac{n}{6k+m/n}+\frac{n}{6+m/n}-\frac{k^2+4k+1}{12k(k+1)}\right).
\end{align*}

Since $k\geq2$, clearly $360k^4+810k^3+809k^2+338k>91$.  Equivalently,
\[864k^4+3456k^3+3924k^2+1332k>504k^4+2646k^3+3115k^2+994k+91\]
and so
\[\frac{72k^2+216k+111}{504k^2+630k+91}>\frac{k^2+4k+1}{12k^2+12k}.\]

Factoring both sides, we obtain
\[\frac{1}{12k+13}+\frac{1}{6k+1}+\frac{1}{7}>\frac{k^2+4k+1}{12k(k+1)}.\]

Moreover, since $n\geq m\geq2$, $0<\frac{m}{n}\leq1$, we have
\begin{align*}
\frac{k^2+4k+1}{12k(k+1)}&<\frac{1}{12k+13}+\frac{1}{6k+1}+\frac{1}{7}\\
&\leq\frac{1}{12k+12+m/n}+\frac{1}{6k+m/n}+\frac{1}{6+m/n}
\end{align*}

Multiplying both sides by $\frac{m}{n}$, we obtain
\[\frac{m}{12(k+1)n+m}+\frac{m}{6kn+m}+\frac{m}{6n+m}>\frac{(k^2+4k+1)m}{12k(k+1)n}.\]

Thus $F$ is increasing in $n$ and is minimum when $n=m$.  When $n=m$, we obtain
\begin{align*}
F&=\frac{1}{12k+13}+\frac{1}{6k+1}+\frac{1}{7}-\frac{k^2+4k+1}{12k(k+1)}\\
&=\frac{360k^4+810k^3+809k^2+338k-91}{6048k^4+13608k^3+8652k^2+1092k}.
\end{align*}

Now since $F$ is decreasing in $k$ and $k\geq2$, $\frac{16061}{242424}\geq F\geq\frac{5}{84}$.  Thus we obtain
\begin{align*}
B_k(n,m)&>\frac{\sqrt{2}m^3(n-2m)}{n(k+1)(kn+n+m)(n+m)(kn+2m)}{\left(\frac{{(k+1)}^{k+1}}{k^k}\right)}^{\frac{n}{2m}}e^F\\
&\geq\frac{\sqrt{2}m^3(n-2m)}{n(k+1)(kn+n+m)(n+m)(kn+2m)}{\left(\frac{{(k+1)}^{k+1}}{k^k}\right)}^{\frac{n}{2m}}e^{\frac{5}{84}}.\qedhere
\end{align*}
\end{proof}

\begin{Lemma}
For $m,n\in\mathbb{N}$ with $3\leq m\leq 7$ and $n\geq 10437$,
\[\prod_{\frac{8n}{m}<p\leq\frac{9n}{m}}\mkern-12mu p<e^{1.129918(\frac{n}{m})}.\]
\end{Lemma}
\begin{proof}
Observe the well-known identity that $0.985x<\vartheta(x)<1.001102x$ where the left-hand inequality holds for $x\geq 11927$, the right-hand side for $x\geq 1$, and where $\vartheta(x)=\sum_{p\leq x}\log p$ is the first Chebyshev function, see \cite{schoenfeld1976}.  That is,
\[e^{0.985x}<\prod_{p\leq x}p<e^{1.001102x}\]
and hence
\[\prod_{\frac{8n}{m}<p\leq\frac{9n}{m}}p<e^{1.001102(\frac{9n}{m})-0.985(\frac{8n}{m})=e^{1.129918(\frac{n}{m})}}.\]

Now since $\frac{8n}{m}\geq 11927$ and $3\leq m\leq 7$, $n\geq \frac{11927m}{8}$.  Thus $n\geq 10437>\frac{11927\cdot 7}{8}$ assures that the inequality holds for $3\leq m\leq 7$.
\end{proof}

Now we may show that $B_8(n,m)$ is an upper bound for the product of prime numbers between $\frac{8n}{m}$ and $\frac{9n}{m}$ as was desired.

\begin{Lemma}
For $m,n\in\mathbb{N}$ with $3\leq m\leq 7$ and $n\geq 10437$,
\[B_8(n,m)>\prod_{\frac{8n}{m}<p\leq\frac{9n}{m}}\mkern-12mu p.\]
\end{Lemma}
\begin{proof}
Let $m,n\in\mathbb{N}$ such that $3\leq m\leq 7$ and $n\geq 10437$.  Clearly the inequality
\begin{align*}
\frac{1}{m}\left(\frac{\log(9^9/8^8)}{2}-1.129918\right)>\frac{\log 2}{2n}&+\frac{3\log 3}{n}+\frac{\log n}{n}\\
&+\frac{\log(9n+m)}{n}+\frac{\log(n+m)}{n}+\frac{\log(4n+m)}{n}
\end{align*}
holds for $n=10437$ and $m=3,4,5,6,$ and $7$.  Moreover, the right-hand side of the inequality is decreasing in $n$ and hence the inequality also holds for all $n\geq 10437$.

Adding $\frac{3\log m}{n}+\frac{\log(n-2m)}{n}+\frac{5}{84n}$ to the left-hand side, we obtain
\begin{align*}
&\frac{1}{m}\left(\frac{\log(9^9/8^8)}{2}-1.129918\right)+\frac{3\log m}{n}+\frac{\log(n-2m)}{n}+\frac{5}{84n}\\
&\hspace{160pt}>\frac{\log 2}{2n}+\frac{3\log 3}{n}+\frac{\log n}{n}+\frac{\log(9n+m)}{n}\\
&\hspace{201pt}+\frac{\log(n+m)}{n}+\frac{\log(4n+m)}{n}.
\end{align*}

Multiplying both sides by $n$ and rearranging terms, we obtain
\begin{align*}
3\log m&+\log(n-2m)-3\log 3-\log n-\log(9n+m)-\log(n+m)\\
&-\frac{\log 2}{2}-\log(4n+m)+\frac{n}{2m}\log(9^9/8^8)+\frac{5}{84}>1.129918(n/m).
\end{align*}

Taking both sides to the base $e$, we obtain
\[\frac{m^3(n-2m)}{9\sqrt{2}n(9n+m)(n+m)(4n+m)}{\left(\frac{9^9}{8^8}\right)}^{\frac{n}{2m}}e^{\frac{5}{84}}>e^{1.129918(\frac{n}{m})}.\]

Now by Lemmas 2.2.3 and 2.2.4, we obtain
\[B_8(n,m)>\frac{m^3(n-2m)}{9\sqrt{2}n(9n+m)(n+m)(4n+m)}{\left(\frac{9^9}{8^8}\right)}^{\frac{n}{2m}}e^{\frac{5}{84}}>e^{1.129918(\frac{n}{m})}>\prod_{\frac{8n}{m}<p\leq\frac{9n}{m}}p\]
as desired.
\end{proof}

We now need to bound the binomial coefficient $\binom{(k+1)n}{kn}$ from below, and $B_k(n,m)$ and $\binom{9n/2}{4n}$ from above for later use in our main result as in the next three lemmas.

\begin{Lemma}
For a positive integer $n\geq 28327$,
\[\binom{9n}{8n}>0.4231409\cdot n^{-1/2}{\left(\frac{9^9}{8^8}\right)}^n .\]
\end{Lemma}
\begin{proof}
By Lemma 2.0.1, we obtain
\begin{align*}
\binom{(k+1)n}{kn}&>\frac{l((k+1)n)}{u(kn)u(n)}\\
&=\sqrt{\frac{k+1}{2\pi kn}}{\left(\frac{{(k+1)}^{k+1}}{k^k}\right)}^{n}e^{\frac{1}{12(k+1)n+1}-\frac{k+1}{12kn}}.
\end{align*}

Now when $k=8$, we obtain
\[\binom{9n}{8n}>\sqrt{\frac{9}{16\pi n}}{\left(\frac{9^9}{8^8}\right)}^{n}e^{\frac{1}{108n+1}-\frac{3}{32n}}.\]

Furthermore, since $\frac{1}{108n+1}-\frac{3}{32n}<0$ and $\frac{1}{108n+1}-\frac{3}{32n}$ is increasing in $n$, $e^{\frac{1}{108n+1}-\frac{3}{32n}}$ is increasing in $n$.  Moreover, since $n\geq 28327$, $e^{\frac{1}{108n+1}-\frac{3}{32n}}\geq e^{-\frac{8271487}{2773160725088}}$.

Thus
\begin{align*}
\binom{9n}{8n}&>\sqrt{\frac{9}{16\pi n}}{\left(\frac{9^9}{8^8}\right)}^{n}e^{\frac{1}{108n+1}-\frac{3}{32n}}\\
&>\sqrt{\frac{9}{16\pi n}}{\left(\frac{9^9}{8^8}\right)}^{n}e^{-\frac{8271487}{2773160725088}}\\
&>0.4231409\cdot n^{-1/2}{\left(\frac{9^9}{8^8}\right)}^{n}.\qedhere
\end{align*}
\end{proof}

\begin{Lemma}
The following inequalities hold for $n\geq 93$:
\begin{enumerate}[font=\normalfont]
\item $B_8(n,3)<0.065661\,n^4\,{\left(\frac{9^9}{8^8}\right)}^{\frac{n}{6}}$.
\item $B_8(n,5)<0.014183\,n^4\,{\left(\frac{9^9}{8^8}\right)}^{\frac{n}{10}}$.
\item $B_8(n,7)<0.005169\,n^4\,{\left(\frac{9^9}{8^8}\right)}^{\frac{n}{14}}$.
\end{enumerate}
\end{Lemma}
\begin{proof}
By Lemma 2.2.2, we have
\[B_8(n,m)<\frac{9n(8n+m)(9n+2m)(n+2m)}{4\sqrt{2}m^3 (n-m)}{\left(\frac{9^9}{8^8}\right)}^{\frac{n}{2m}}e^{\frac{12029}{111150}}.\]

Now consider
\[\frac{n(8n+m)(9n+2m)(n+2m)}{n-m}<n^4,\]
equivalently
\[72n^4+169n^3 m+52n^2 m^2+4nm^3<n^5-n^4 m.\]

Dividing both sides by $n^5$ and rearranging terms, we obtain
\[\frac{72}{n}+\frac{169m}{n^2}+\frac{52m^2}{n^3}+\frac{4m^3}{n^4}+\frac{m}{n}<1.\]

Now the left-hand side is decreasing in $n$ and it suffices to verify the case when $n=93$ for the respective values of $m$.  When $n=93$, we obtain
\[\frac{799450}{923521}<1,\quad\frac{69365294}{74805201}<1,\quad\text{and }\frac{74014306}{74805201}<1\vspace{-10pt}\]
when $m=3$, $m=5$, and $m=7$, respectively.

Thus
\[\frac{n(8n+m)(9n+2m)(n+2m)}{n-m}<n^4\]
and the results follow directly.
\end{proof}

\begin{Lemma}
For $n$ a positive integer,
\[{9n/2\brace 4n}<2.692861\sqrt{n}{\left(\frac{9^9}{8^8}\right)}^{\frac{n}{2}}.\vspace{-8pt}\]
\end{Lemma}
\begin{proof}
By Lemma 2.0.1, we obtain
\begin{align*}
{9n/2\brace 4n}&\leq\frac{9n}{2}\binom{\left[\frac{9n}{2}\right]}{4n}<\frac{9n}{2}\cdot\frac{u\left(\left[\frac{9n}{2}\right]\right)}{l(4n)l\left(\left[\frac{9n}{2}\right]-4n\right)}\\
&\leq\frac{9n}{2}\cdot\frac{u\left(\frac{9n}{2}\right)}{l(4n)l\left(\frac{n}{2}\right)}=\frac{9n}{2}\sqrt{\frac{9}{8\pi n}}{\left(\frac{9^9}{8^8}\right)}^{\frac{n}{2}}e^{\frac{1}{54n}-\frac{1}{6n+1}-\frac{1}{48n+1}}
\end{align*}

Now since $\frac{1}{54n}-\frac{1}{6n+1}-\frac{1}{48n+1}<0$ and $\frac{1}{54n}-\frac{1}{6n+1}-\frac{1}{48n+1}$ is increasing in $n$, $e^{\frac{1}{54n}-\frac{1}{6n+1}-\frac{1}{48n+1}}$ is increasing in $n$.  Hence $e^{\frac{1}{54n}-\frac{1}{6n+1}-\frac{1}{48n+1}}\leq1$ and so
\[\binom{9n/2}{4n}<\frac{9n}{2}\sqrt{\frac{9}{8\pi n}}{\left(\frac{9^9}{8^8}\right)}^{\frac{n}{2}}<2.692861\sqrt{n}{\left(\frac{9^9}{8^8}\right)}^{\frac{n}{2}}.\qedhere\]
\end{proof}

\begin{Lemma}
For all $n\geq 56833$,
\[{\left(\frac{9^9}{8^8}\right)}^{\frac{17n}{105}}>e^{1.001102\left(\frac{9n}{19}\right)}n^{13}{(9n)}^{\frac{7.53036\sqrt{n}}{\log(9n)}}.\]
\end{Lemma}
\begin{proof}
The following inequalities are equivalent for all $n\geq 56833$:
\[{\left(\frac{9^9}{8^8}\right)}^{\frac{17n}{105}}>e^{1.001102\left(\frac{9n}{19}\right)}n^{13}{(9n)}^{\frac{7.53036\sqrt{n}}{\log(9n)}}\]
\[\frac{17n}{105}\log\left(\frac{9^9}{8^8}\right)-1.001102\left(\frac{9n}{19}\right)>13\log n+7.53036\sqrt{n}\]
\[\frac{17}{105}\log\left(\frac{9^9}{8^8}\right)-1.001102\left(\frac{9}{19}\right)>\frac{13\log n}{n}+\frac{7.53036}{\sqrt{n}}.\]

Now the right-hand side is decreasing in $n$, so it suffices to verify the case when $n=56833$.  When $n=56833$, we obtain
\[\frac{17}{105}\log\left(\frac{9^9}{8^8}\right)>0.0340918>\frac{13\log 56833}{56833}+\frac{7.53036}{\sqrt{56833}}.\qedhere \]
\end{proof}

\begin{Theorem}
For any positive integer $n>4$ there is a prime number between $8n$ and $9n$.
\end{Theorem}
\begin{proof}
It can be easily verified that for $n=5, 6, \ldots, 56832$ there is always a prime between $8n$ and $9n$.  Now let $n\geq 56833$ and consider:
\[\binom{9n}{8n}=\frac{(8n+1)(8n+2)\cdots(9n)}{1\cdot2\cdots n}.\]

The product of primes between $8n$ and $9n$, if there are any, divides $\binom{9n}{8n}$.  Following the notation used in \cite{bachraoui2006,erdos1932,loo2011}, we let
 \[T_{1}={\displaystyle\prod_{p\leq \sqrt{9n}}\mkern-1mu p^{\beta(p)}},\quad T_{2}={\displaystyle\prod_{\sqrt{9n} < p\leq 8n}\mkern-14mu p^{\beta(p)}},\quad\text{and } T_{3}={\displaystyle\prod_{8n+1\leq p\leq 9n}\mkern-16mu p}\]
such that
\[\binom{9n}{8n}=T_{1}T_{2}T_{3}.\]

The prime decomposition of $\binom{9n}{8n}$ implies that the powers in $T_2$ are less than 2, see \cite[p.~24]{erdos2003} for the prime decomposition of $\binom{n}{j}$.  In addition, the prime decomposition of $\binom{9n}{8n}$ yields the upper bound for $T_{1}$:
\[T_{1}<{(9n)}^{\pi(\sqrt{9n})}.\]
See \cite[p.~24]{erdos2003}.  But $\pi(x)\leq\frac{1.25506x}{\log x}$, see \cite{rosser1962}.  Thus
\[T_{1}<{(9n)}^{\pi(\sqrt{9n})}\leq{(9n)}^{\frac{2.51012\sqrt{9n}}{\log(9n)}}.\]

By Lemma 2.2.1, we know that if $n,m\in\mathbb{N}$, $n\geq m\geq2$, and $p$ is a prime number such that $\frac{8n}{m}<p\leq\frac{9n}{m}$, then $p\mid B_8(n,m)$.  So let $m$ satisfy $7\geq m\geq 3$, let $A={9n/2\brace 4n}$, and observe the following for a prime number $p$ in $T_2$:

\noindent$\circ$ If $\frac{9n}{2}<p\leq8n$, then
\[n<p\leq8n<9n<2p.\]
Hence $\beta(p)=0$.

\noindent$\circ$ If $4n<p\leq\frac{9n}{2}$, then $p$ divides $A$.\\
Hence $\displaystyle\prod_{\frac{9n}{2}<p\leq4n}\mkern-10mu p$ divides $A$.

\noindent$\circ$ If $3n<p\leq4n$, then
\[n<p<2p\leq8n<9n<3p.\]
Hence $\beta(p)=0$.

\noindent$\circ$ If $\frac{8n}{3}<p\leq3n$, then $p$ divides $B_8(n,3)$.\\
Hence $\displaystyle\prod_{\frac{8n}{3}<p\leq3n}\mkern-10mu p$ divides $B_8(n,3)$.

\noindent$\circ$ If $\frac{9n}{4}<p\leq\frac{8n}{3}$, then
\[n<p<3p\leq8n<9n<4p.\]
Hence $\beta(p)=0$.

\noindent$\circ$ If $2n<p\leq\frac{9n}{4}$, then
\[\frac{n}{2}<2n<p\leq\frac{9n}{4}<4n<2p\leq\frac{9n}{2}<3p.\]
Hence $\displaystyle\prod_{2n<p\leq\frac{9n}{4}}\mkern-10mu p$ divides $A$.

\noindent$\circ$ If $\frac{9n}{5}<p\leq2n$, then
\[n<p<4p\leq8n<9n<5p.\]
Hence $\beta(p)=0$.

\noindent$\circ$ If $\frac{8n}{5}<p\leq\frac{9n}{5}$, then $p$ divides $B_8(n,5)$.\\
Hence $\displaystyle\prod_{\frac{8n}{5}<p\leq\frac{9n}{5}}\mkern-10mu p$ divides $B_8(n,5)$.

\vspace{18pt}

\noindent$\circ$ If $\frac{3n}{2}<p\leq\frac{8n}{5}$, then
\[n<p<5p\leq8n<9n<6p.\]
Hence $\beta(p)=0$.

\vspace{18pt}

\noindent$\circ$ If $\frac{4n}{3}<p\leq\frac{3n}{2}$, then
\[\frac{n}{2}<\frac{4n}{3}<p\leq\frac{3n}{2}<2p<4n<3p\leq\frac{9n}{2}<4p.\]
Hence $\displaystyle\prod_{\frac{4n}{3}<p\leq\frac{3n}{2}}\mkern-10mu p$ divides $A$.

\vspace{18pt}

\noindent$\circ$ If $\frac{9n}{7}<p\leq\frac{4n}{3}$, then
\[n<p<6p\leq8n<9n<7p.\]
Hence $\beta(p)=0$.

\vspace{18pt}

\noindent$\circ$ If $\frac{8n}{7}<p\leq\frac{9n}{7}$, then $p$ divides $B_8(n,7)$.\\
Hence $\displaystyle\prod_{\frac{8n}{7}<p\leq\frac{9n}{7}}\mkern-10mu p$ divides $B_8(n,7)$.

\vspace{18pt}

\noindent$\circ$ If $\frac{9n}{8}<p\leq\frac{8n}{7}$, then
\[n<p<7p\leq8n<9n<8p.\]
Hence $\beta(p)=0$.

\newpage

\noindent$\circ$ If $n<p\leq\frac{9n}{8}$, then
\[\frac{n}{2}<n<p<3p<4n<4p\leq\frac{9n}{2}<5p.\]
Hence $\displaystyle\prod_{n<p\leq\frac{9n}{8}}\mkern-6mu p$ divides $A$.

\noindent$\circ$ If $\frac{9n}{10}<p\leq n$, then
\[p\leq n<2p<8p\leq8n<9p\leq9n<10p.\]
Hence $\beta(p)=0$.

\noindent$\circ$ If $\frac{4n}{5}<p\leq\frac{9n}{10}$, then
\[\frac{n}{2}<p<4p<4n<5p\leq\frac{9n}{2}<6p.\]
Hence $\displaystyle\prod_{\frac{4n}{5}<p\leq\frac{9n}{10}}\mkern-10mu p$ divides $A$.

\noindent$\circ$ If $\frac{3n}{4}<p\leq\frac{4n}{5}$, then
\[p<n<2p<10p\leq8n<11p<9n<12p.\]
Hence $\beta(p)=0$.

\noindent$\circ$ If $\frac{2n}{3}<p\leq\frac{3n}{4}$, then
\[\frac{n}{2}<p<5p<4n<6p\leq\frac{9n}{2}<7p.\]
Hence $\displaystyle\prod_{\frac{2n}{3}<p\leq\frac{3n}{4}}\mkern-10mu p$ divides $A$.

\noindent$\circ$ If $\frac{9n}{14}<p\leq\frac{2n}{3}$, then
\[p<n<2p<12p\leq8n<13p<9n<14p.\]
Hence $\beta(p)=0$.

\noindent$\circ$ If $\frac{4n}{7}<p\leq\frac{9n}{14}$, then
\[\frac{n}{2}<p<6p<4n<7p\leq\frac{9n}{2}<8p.\]
Hence $\displaystyle\prod_{\frac{4n}{7}<p\leq\frac{9n}{14}}\mkern-10mu p$ divides $A$.

\noindent$\circ$ If $\frac{9n}{16}<p\leq\frac{4n}{7}$, then
\[p<n<2p<14p\leq8n<15p<9n<16p.\]
Hence $\beta(p)=0$.

\noindent$\circ$ If $\frac{n}{2}<p\leq\frac{9n}{16}$, then
\[\frac{n}{2}<p<7p<4n<8p\leq\frac{9n}{2}<9p.\]
Hence $\displaystyle\prod_{\frac{n}{2}<p\leq\frac{9n}{16}}\mkern-8mu p$ divides $A$.

\noindent$\circ$ If $\frac{9n}{19}<p\leq\frac{n}{2}$, then
\[p<2p\leq n<3p<16p\leq8n<17p<18p\leq9n<19p.\]
Hence $\beta(p)=0$.

By the fact that $\displaystyle\prod_{p\leq x}p<e^{1.001102x}$ as shown in \cite{schoenfeld1976}, we obtain\\
$\circ\displaystyle\prod_{\sqrt{9n}<p\leq\frac{9n}{19}}\mkern-16mu p\leq\displaystyle\prod_{p\leq\frac{9n}{19}}p<e^{1.001102(\frac{9n}{19})}.$

By Lemmas 2.2.7 and 2.2.8, we obtain
\begin{align*}
T_{2}&<e^{1.001102(\frac{9n}{19})}{9n/2\brace 4n}B_8(n,3)B_8(n,5)B_8(n,7)\\
&<0.000013e^{1.001102(\frac{9n}{19})}n^{\frac{25}{2}}{\left(\frac{9^9}{8^8}\right)}^{\frac{n}{2}+\frac{n}{6}+\frac{n}{10}+\frac{n}{14}}\\
&=0.000013e^{1.001102(\frac{9n}{19})}n^{\frac{25}{2}}{\left(\frac{9^9}{8^8}\right)}^{\frac{88n}{105}}.
\end{align*}

By Lemma 2.2.6,\vspace{12pt}
\[\binom{9n}{8n}>0.4231409n^{-1/2}{\left(\frac{9^9}{8^8}\right)}^{n}.\]

Thus we obtain
\begin{align*}
T_3&=\binom{9n}{8n}\frac{1}{T_1 T_2}\\
&>\frac{0.4231409}{0.000013}e^{-1.001102(\frac{9n}{19})}n^{-13}{\left(\frac{9^9}{8^8}\right)}^{\frac{17n}{105}}{(9n)}^{-\frac{2.51012\sqrt{9n}}{\log(9n)}}\\
&>e^{-1.001102(\frac{9n}{19})}n^{-13}{\left(\frac{9^9}{8^8}\right)}^{\frac{17n}{105}}{(9n)}^{-\frac{2.51012\sqrt{9n}}{\log(9n)}}\\
&>1,
\end{align*}
where the last inequality follows by Lemma 2.2.9.  Consequently the product $T_3$ of prime numbers between $8n$ and $9n$ is greater than $1$ and therefore the existence of such numbers is proven.
\end{proof}

With the proof of the previous theorem complete we may also show that there is always a prime number between $n$ and $\frac{9n+63}{8}$ for any positive integer $n$ as in the following theorem.

\begin{Theorem}
For any positive integer $n$ there is a prime number between $n$ and $\frac{9n+63}{8}$.
\end{Theorem}
\begin{proof}
The cases when $n\in\{1,2,3,4\}$ may be verified directly.  Now let $n\geq 5$ be a positive integer.  By the division algorithm $8\mid(n+r)$ for some $r\in\{0,1,2,3,4,5,6,7\}$.  By Theorem 2.2.10 there exists a prime number $p$ such that
\[p\in\left(n+r,\frac{9(n+r)}{8}\right).\]
Since $\left(n+r,\frac{9(n+r)}{8}\right)$ is contained in $\left(n,\frac{9n+63}{8}\right)$ for all $r$ and $n$, $p\in\left(n,\frac{9n+63}{8}\right)$ as desired.
\end{proof}

Since we, in essence, skipped the cases when $k=5,6,$ and $7$ in this chapter, we will show that they readily follow as corollaries of Theorem 2.2.10.

\begin{Corollary}
For any positive integer $n>1$ there is a prime number between $5n$ and $6n$.
\end{Corollary}

\vspace{-12pt}
\begin{proof}
The cases for $2\leq n\leq 61$ may be verified directly.  Let $n\geq 62$.  By the division algorithm $n=8k+j$ for some $k\in\mathbb{N}$ and $j\in\{0,1,\ldots,7\}$.  Observe that
\vspace{-18pt}
\[40k+5j\leq40k+8j<45k+9j\leq48k+6j\]

\vspace*{-16pt}
\noindent and as a consequence

\vspace{-20pt}
\[(8(5k+j),9(5k+j))\subset(5(8k+j),6(8k+j))=(5n,6n).\]

\vspace{-2pt}
By Theorem 2.2.10, $p\in(8(5k+j),9(5k+j))$ and therefore $p\in(5n,6n)$.
\end{proof}

\begin{Corollary}
For any positive integer $n>4$ there is a prime number between $6n$ and $7n$.
\end{Corollary}

\vspace{-18pt}
\begin{proof}
The cases for $5\leq n\leq 62$ may be verified directly.  Let $n\geq 63$.  By the division algorithm $n=8k+j$ for some $k\in\mathbb{N}$ and $j\in\{0,1,\ldots,7\}$.  Observe that
\[48k+6j\leq48k+8j<54k+9j\leq56k+7j\]
and as a consequence
\[(8(6k+j),9(6k+j))\subset(6(8k+j),7(8k+j))=(6n,7n).\]

By Theorem 2.2.10, $p\in(8(6k+j),9(6k+j))$ and therefore $p\in(6n,7n)$ as desired.
\end{proof}

\begin{Corollary}
For any positive integer $n>2$ there is a prime number between $7n$ and $8n$.
\end{Corollary}

\vspace{-18pt}
\begin{proof}
The cases for $3\leq n\leq 63$ may be verified directly.  Let $n\geq 64$.  By the division algorithm $n=8k+j$ for some $k\in\mathbb{N}$ and $j\in\{0,1,\ldots,7\}$.  Observe that
\[56k+7j\leq56k+8j<63k+9j\leq64k+8j\]
and as a consequence
\[(8(7k+j),9(7k+j))\subset(7(8k+j),8(8k+j))=(7n,8n).\]

By Theorem 2.2.10, $p\in(8(7k+j),9(7k+j))$ and therefore $p\in(7n,8n)$ as desired.
\end{proof}
\section{Primes in the Interval [519$n$, 520$n$]}
\label{sec:Primes519n520n}

We will now move away from the elementary methods used in the previous two sections and move towards an analytic number theory approach to establish an improved result.  Our proof conforms to S. Ramanujan's \cite{ramanujan1919} proof of Bertrand's postulate.  It also conforms with J. Nagura's \cite{nagura1952} proof of primes in the interval $n$ to $\frac{6n}{5}$.

The basis of our proof is to approximate the first Chebyshev function $\vartheta$ using the second Chebyshev function $\psi$.  That is, the functions:
\[\psi(x)=\sum_{m=1}^{\infty}\vartheta(x^{1/m}) \,\,\,\,\,\text{and}\,\,\,\,\, \vartheta(x)=\sum_{\substack{p\leq x\\p\text{ prime}}}\log p.\]

If we are able to show that $\vartheta(\frac{520n}{519})-\vartheta(n)>0$, then by taking both sides to the base $e$, we obtain:
\[\prod_{\substack{p\leq \frac{520n}{519}\\ p\text{ prime}}}\mkern-4mu p>\prod_{\substack{p\leq n\\ p\text{ prime}}}\mkern-4mu p.\]
However, then the product of primes between $n$ and $\frac{520n}{519}$ is greater than 1, and so there must be at least one prime number between $n$ and $\frac{520n}{519}$.

In 1976, L. Schoenfeld \cite{schoenfeld1976} showed that for all $x\geq e^{19}$, the upper and lower bounds of $\psi(x)$ are given by the inequality $0.99903839x<\psi(x)<1.00096161x$.  In his paper he achieved these approximations by using analytic methods to show that $\mid\psi(x)-x\mid<0.00096161x$ from which the double inequality follows.  We will use this approximation in the following theorem.

\begin{Theorem}
For $n\geq 31409$, there exists at least one prime number $p$ such that $n<p<\frac{520n}{519}$.
\end{Theorem}
\begin{proof}
In order to prove $\vartheta(\frac{520n}{519})-\vartheta(n)>0$ for the values of $n$ as small as possible, let us use

\vspace{-40pt}
\begin{equation}
\psi(x)-\psi(x^{1/2})-\psi(x^{1/3})-\cdots-\psi(x^{1/503})=\psi(x)-\sum_{p\leq 503}\psi(x^{1/p})\geq\vartheta(x)
\end{equation}

\vspace{-20pt}
\noindent and

\vspace{-40pt}
\begin{equation}
\vartheta(x)\geq\psi(x)-\psi(x^{1/2})-\psi(x^{1/3})-\cdots-\psi(x^{1/509})=\psi(x)-\sum_{p\leq 509}\psi(x^{1/p}).
\end{equation}
Note that we chose 503 and 509 as the summation upper limit in the summations of equations 2.1 and 2.2, respectively.  These choices are the second largest prime and the largest prime less than 519, respectively.

Thus we obtain
\[\vartheta\left(\frac{520n}{519}\right)-\vartheta(n)\geq \psi\left(\frac{520n}{519}\right)-\sum_{p\leq 509}\psi\left({\left(\frac{520n}{519}\right)}^{1/p}\right)-\psi(n)+\sum_{p\leq 503}\psi(n^{1/p}).\]

By using the approximation for $\psi(x)$,
\[0.99903839x<\psi(x)<1.00096161x,\]
we obtain
\begin{align*}
\vartheta\left(\frac{520n}{519}\right)-\vartheta(n)>&\,0.99903839\left(\frac{520n}{519}+\sum_{p\leq503}n^{1/p}\right)\\
&-1.00096161\left(n+\sum_{p\leq509}{\left(\frac{520n}{519}\right)}^{1/p}\right)
\end{align*}
which is positive for $n\geq e^{19}$.

However, we may verify the cases for $31409\leq n\leq e^{19}$ using a program such as Mathematica and our theorem is thus proved.
\end{proof}

From the previous theorem we may prove a corollary that $\vartheta(\frac{(k+1)n}{k})-\vartheta(n)>0$ for all $k$ and $n$ such that $n\geq 31409$ and $519\geq k\geq 1$.  We will also apply this corollary to show that there is a prime number between $519n$ and $520n$ for all $n\geq 15$.  This theorem also shows that there is always a prime between $kn$ and $(k+1)n$ for all $n\geq k$ and $519\geq k\geq 2$ which is a significant improvement in the number of cases for $k$ that we were able to show in Sections 2.1 and 2.2.

\begin{Corollary}
For $k,n\in\mathbb{N}$ with $n\geq 31409$ and $519\geq k\geq 1$,
\[\vartheta\left(\frac{(k+1)n}{k}\right)-\vartheta(n)>0.\]
\end{Corollary}
\begin{proof}
The inequality follows directly by noting that for all $n\geq 31409$,
\[\vartheta(2n)\geq\vartheta(3n/2)\geq\vartheta(4n/3)\geq\ldots\geq\vartheta(519n/518)\geq\vartheta(520n/519).\]
Hence
\[\vartheta(2n)-\vartheta(n)\geq\vartheta(3n/2)-\vartheta(n)\geq\ldots\geq\vartheta(520n/519)-\vartheta(n)>0\]
where the last inequality follows by Theorem 2.3.1.
\end{proof}

\begin{Theorem}
For $n\geq 15$ there is a prime number between $519n$ and $520n$.
\end{Theorem}
\begin{proof}
The cases for $15\leq n\leq 31408$ may be verified directly.  Now let $n\geq 31409$ and consider $\vartheta(\frac{520n}{519})-\vartheta(n)>0$ as shown in Corollary 2.3.2.  Taking both sides of the inequality to the base $e$, we obtain
\[e^{\vartheta(\frac{520n}{519})-\vartheta(n)}=\prod_{n<p\leq\frac{520n}{519}}\mkern-10mu p>1.\]

Therefore there exists at least one prime number between $n$ and $\frac{520n}{519}$.  Allowing $n=519m$ for some $m\in\mathbb{N}$, we deduce that there exists at least one prime number between $519m$ and $520m$ as desired.
\end{proof}

\begin{Theorem}
For $k,n\in\mathbb{N}$ with $n\geq k$ and $519\geq k\geq 2$, there is at least one prime number between $kn$ and $(k+1)n$.
\end{Theorem}
\begin{proof}
For any choice of $k$ such that $2\leq k\leq 519$, the cases for $k\leq n\leq 31408$ may be verified directly.  By Theorem 2.3.1, for all $n\geq 31409$, $\pi(\frac{520n}{519})-\pi(n)\geq 1$.  Now,
\[\pi(2n)\geq\pi(3n/2)\geq\pi(4n/3)\geq\ldots\geq\pi(520n/519)\geq\pi(n)+1.\]
Therefore,
\begin{equation}
\pi(2n)-\pi(n)\geq\pi(3n/2)-\pi(n)\geq\ldots\geq\pi(520n/519)-\pi(n)\geq1.
\end{equation}
Now by letting $n=m$, $n=2m$, $\ldots$, $n=519m$ for some $m\in\mathbb{N}$ in the respective inequalities above, we obtain $\pi(2m)-\pi(m)\geq 1$, $\pi(3m)-\pi(2m)\geq 1$, $\ldots$, $\nolinebreak{\pi(520m)-\pi(519m)\geq 1}$ as desired.
\end{proof}

From equation (2.3) in the previous theorem we obtain a direct corollary.
\begin{Corollary}
For $k,n\in\mathbb{N}$ with $n\geq k$ and $519\geq k\geq 2$, there is at least one prime number between $n$ and $\frac{(k+1)n}{k}$.
\end{Corollary}
\section{Primes between $n^2$ and $(n+1)^{2+\varepsilon}$}
\label{sec:PrimesNsquared}
If we allow $k=n$ in our previous question of a prime number $p$ satisfying $\nolinebreak{kn<p<(k+1)n}$, then we obtain a sharper inequality than Legendre's conjecture.  Legendre's conjecture states that for every positive integer $n$ there exists a prime number between $n^2$ and $(n+1)^2$.  That is, there is always a prime number between any two successive perfect squares.  This question is one of the famous Landau problems proposed by Edmund Landau in 1912 at the International Congress of Mathematicians, see \cite{hardy1979}.

While this question remains open to date, some progress has been made for sufficiently large $n$.  Perhaps most notably, Chen Jingrun \cite{chen1975} has shown that there exists a number $p$ satisfying $\nolinebreak{n^2<p<(n+1)^2}$ such that $p$ is either a prime number or a semiprime; where a semiprime is a product of two prime numbers, not necessarily distinct.  Furthermore, there is always a prime number between $n-n^{\theta}$ and $n$ for $\theta=23/42$ and $\theta=0.525$, see \cite[p.~415]{hardy1979}, \cite{iwaniec1984}, and \cite{baker2001}.

While we offer no proof of Legendre's conjecture, we do show that for all positive integers $n$ there exists a prime number $p$ such that $n^2<p<{(n+1)}^{2+\varepsilon}$ for $\nolinebreak{\varepsilon=0.00011516865557559264}$ and $\varepsilon=0.000001$.

\vspace{8pt}
\begin{Theorem}
For $n$ a positive integer, there exists a prime number between $n^2$ and ${(n+1)}^{2+\varepsilon}$ where $\varepsilon=0.00011516865557559264$.
\end{Theorem}

\vspace{-8pt}
\begin{proof}
The cases for $n=1, 2, \ldots, 4407$ may be verified directly.  By Theorem 5.2 of \cite{dusart2010} we have for all $x\geq3594641$:
\[\mid\vartheta(x)-x\mid<\frac{0.2x}{\log^2 x}.\]

Let $n\geq 4408$ so that ${(n+1)}^{2+\varepsilon}>n^2>3594641$.  Now
\[\vartheta({(n+1)}^{2+\varepsilon})-\vartheta(n^2)>\left({(n+1)}^{2+\varepsilon}-\frac{0.2{(n+1)}^{2+\varepsilon}}{\log^2 {(n+1)}^{2+\varepsilon}}\right)-\left(n^2+\frac{0.2n^2}{\log^2 n^2}\right).\]

Consider the following equivalent inequalities for $n\geq 4408$:
\[{(n+1)}^{2+\varepsilon}>n^2+\frac{0.2{(n+1)}^{2+\varepsilon}}{\log^2{(n+1)}^{2+\varepsilon}}+\frac{0.2n^2}{\log^2 n^2},\]
\[1>{\left(\frac{n}{n+1}\right)}^2 \frac{1}{{(n+1)}^{\varepsilon}}+\frac{0.2}{{(2+\varepsilon)}^{2}\log^2(n+1)}+{\left(\frac{n}{n+1}\right)}^2\frac{0.05}{{(n+1)}^{\varepsilon}\log^2 n}.\]

Now the right-hand side is decreasing in $n$ and so it suffices to verify the case when $n=4408$.  When $n=4408$, we obtain
\[1>{\left(\frac{4408}{4409}\right)}^2 \frac{1}{{4409}^{\varepsilon}}+\frac{0.2}{{(2+\varepsilon)}^{2}\log^2 4409}+{\left(\frac{4408}{4409}\right)}^2\frac{0.05}{{4409}^{\varepsilon}\log^2 4408}.\]

Therefore
\[\left({(n+1)}^{2+\varepsilon}-\frac{0.2{(n+1)}^{2+\varepsilon}}{\log^2{(n+1)}^{2+\varepsilon}}\right)-\left(n^2+\frac{0.2n^2}{\log^2 n^2}\right)>0\]
and hence $\vartheta({(n+1)}^{2+\varepsilon})-\vartheta(n^2)>0$ as desired.
\end{proof}

\vspace{-8pt}
Although similar results may be shown for any $\varepsilon>0$, we shall see that as $\varepsilon\rightarrow0$ the number of base cases for $n$ which must be verified directly increases. In the next theorem we show that $\varepsilon=0.000001$ is sufficient, however this increases the number of base cases which must be verified by 26,010,188.  As mentioned previously, this increase in the number of base cases to verify is expected due to the techniques used within the proofs of this chapter.

\begin{Theorem}
For $n$ a positive integer, there exists a prime number between $n^2$ and ${(n+1)}^{2.000001}$.
\end{Theorem}

\vspace{-18pt}
\begin{proof}
The cases for $n=1, 2, \ldots, 26014595$ may be verified directly.  By Theorem 5.2 of \cite{dusart2010} we have for all $x\geq 7713133853$,
\[\mid\vartheta(x)-x\mid<\frac{0.01x}{\log^2 x}.\]

Let $n\geq 26014596$ so that ${(n+1)}^{2.000001}>n^2>7713133853$.  Now
\[\vartheta({(n+1)}^{2.000001})-\vartheta(n^2)>\left({(n+1)}^{2.000001}-\frac{0.2{(n+1)}^{2.000001}}{\log^2{(n+1)}^{2.000001}}\right)-\left(n^2+\frac{0.2n^2}{\log^2 n^2}\right).\]

Consider the following equivalent inequalities for $n\geq 26014596$:
\[{(n+1)}^{2.000001}>n^2+\frac{0.01{(n+1)}^{2.000001}}{\log^2 {(n+1)}^{2.000001}}+\frac{0.01n^2}{\log^2 n^2},\]
\begin{align*}
1>{\left(\frac{n}{n+1}\right)}^2 \frac{1}{{(n+1)}^{0.000001}}&+\frac{0.01}{{2.000001}^{2}\log^2(n+1)}\\
&+{\left(\frac{n}{n+1}\right)}^2\frac{0.0025}{{(n+1)}^{0.000001}\log^2 n}.
\end{align*}

Now the right-hand side is decreasing in $n$ and so it suffices to verify the case when $n=26014596$.  When $n=26014596$, we obtain
\begin{align*}
1>{\left(\frac{26014596}{26014597}\right)}^2 \frac{1}{{26014597}^{0.000001}}&+\frac{0.01}{{2.000001}^{2}\log^2 26014597}\\
&+{\left(\frac{26014596}{26014597}\right)}^2\frac{0.05}{{26014597}^{0.000001}\log^2 26014596}.
\end{align*}

Therefore
\[{(n+1)}^{2.000001}-\frac{0.2{(n+1)}^{2.000001}}{\log^2 {(n+1)}^{2.000001}}-n^2-\frac{0.2n^2}{\log^2 n^2}>0\]
and hence $\vartheta({(n+1)}^{2.000001})-\vartheta(n^2)>0$ as desired.
\end{proof}
\chapter{The Number of Prime Numbers\label{chapter3}}
In this chapter we consider the question:
\begin{center}
How many prime numbers are there between $n$ and $kn$?
\end{center}

For instance, Bertrand's postulate states that there is at least one prime number between $n$ and $2n$ for all $n>1$.  Moreover, by M. El Bachraoui \cite{bachraoui2006}, we know there exists a prime number between $2n$ and $3n$ for all $n>1$.  Therefore, there are at least two prime numbers between $n$ and $3n$ for $n>1$.

Furthermore, if we take into account A. Loo's \cite{loo2011} result that there is a prime number between $3n$ and $4n$ for $n>1$, then there are three prime numbers between $n$ and $4n$ for all $n>1$.

In the following sections we will extend these methods to improve upon the number of primes between $n$ and $kn$ using our previous results.

\section{Between $n$ and $5n$}
\label{sec:BetweenNand5N}
In the following theorem we extend the previous facts by showing that there are at least four prime numbers between $n$ and $5n$.  In order to accomplish this task we will use Theorem 2.1.4; that is, there is a prime number between $n$ and $(5n+15)/4$ for all $n>2$.

\begin{Theorem}
For any positive integer $n>2$, there are at least four prime numbers between $n$ and $5n$.
\end{Theorem}

\vspace{-18pt}
\begin{proof}
The cases when $n=3, 4, \ldots, 14$ may be verified directly.  Now let $n\geq15$ and by Theorem 2.1.4 we know there exists prime numbers $p_1$, $p_2$, and $p_3$ such that $n<p_1<\frac{5n+15}{4}$, $2n<p_2<\frac{10n+15}{4}$, $3n<p_3<\frac{15n+15}{4}$, and by Theorem 2.1.3 there exist a prime number $p_4$ such that $4n<p_4<5n$.  Hence
\vspace{5pt}
\begin{align*}
n<p_1<\frac{5n+15}{4}<2n<p_2&<\frac{10n+15}{4}\\
&<3n<p_3<\frac{15n+15}{4}\leq4n<p_4<5n.\qedhere
\end{align*}
\end{proof}

In the next theorem we improve on the number of prime numbers between $n$ and $5n$ to show that there are at least seven prime numbers between $n$ and $5n$ for $n>5$.

\vspace{6pt}
\begin{Theorem}
For all $n>5$ there are at least seven prime numbers between $n$ and $5n$.
\end{Theorem}

\vspace{-18pt}
\begin{proof}
The cases when $n=6,7,\ldots, 244$ may be verified directly.  Now let $\nolinebreak{f(n)=\frac{5n+15}{4}}$ for $n\geq 245$ and let $f^m(n)=f(f^{m-1}(n))$.  By Theorem 2.1.4 there exists a prime number between $n$ and $f(n)$.  Furthermore, there exists a prime number between $f^{m-1}(n)$ and $f^m(n)$ for all $\nolinebreak{m\in\mathbb{N}\setminus\{1\}}$.  In general,
\[f^m(n)=\frac{1}{4^m}\left(5^m n+3\displaystyle\sum_{k=0}^{m-1}5^{m-k}\,4^k\right).\]

Consider
\vspace{12pt}
\[f^m(n)=\frac{1}{4^m}\left(5^m n+3\displaystyle\sum_{k=0}^{m-1}5^{m-k}\,4^k\right)\leq 5n.\]

Solving for $n$, we obtain
\[n\geq\frac{3}{5\cdot 4^m-5^m}\displaystyle\sum_{k=0}^{m-1}5^{m-k}\,4^k.\]

However, $5\cdot 4^m-5^m$ is positive only for $m\leq 7$.  So let $m=7$, then for
\[n\geq 245>\frac{3}{5\cdot 4^7-5^7}\sum_{k=0}^{6}5^{7-k}\,4^k\]
there are at least seven prime numbers between $n$ and $5n$.
\end{proof}
\section{Between $n$ and $9n$}
\label{sec:BetweenNand9N}

Similar to the previous section, we may show that there are at least $8$ prime numbers between $n$ and $9n$ for all positive integers $n>2$.

\vspace{6pt}
\begin{Theorem}
For any positive integer $n>2$ there are at least eight prime numbers between $n$ and $9n$.
\end{Theorem}
\begin{proof}
The cases for $n=3,4,\ldots,63$ may be verified directly.  Now let $n\geq64$ and observe that by Theorem 2.2.11 there exists prime numbers $p_1,p_2,\ldots,p_8$ such that
\begin{align*}
n&<p_1<\frac{9n+63}{8}<2n<p_2<\frac{18n+63}{8}<3n<p_3<\frac{27n+63}{8}<4n\\
&<p_4<\frac{36n+63}{8}<5n<p_5<\frac{45n+63}{8}<6n<p_6<\frac{54n+63}{8}<7n\\
&<p_7<\frac{63n+63}{8}<8n<p_8<9n
\end{align*}
where $8n<p_8<9n$ by Theorem 2.2.10.  Therefore there are at least eight prime numbers between $n$ and $9n$ for any positive integer $n>2$.
\end{proof}

In the next theorem we improve on the number of prime numbers between $n$ and $9n$ to show that there are at least eighteen prime numbers between $n$ and $9n$ for $n>8$.

\vspace{6pt}
\begin{Theorem}
For all $n>8$ there are at least eighteen prime numbers between $n$ and $9n$.
\end{Theorem}
\begin{proof}
The cases when $n=9, 10, \ldots, 691$ may be verified directly.  Now let $\nolinebreak{f(n)=\frac{9n+63}{8}}$ for $n\geq 692$ and let $f^m(n)=f(f^{m-1}(n))$.  By Theorem 2.2.11 there exists a prime number between $n$ and $f(n)$.  Furthermore, there exists a prime number between $f^{m-1}(n)$ and $f^m(n)$ for all $\nolinebreak{m\in\mathbb{N}\setminus\{1\}}$.  In general,
\[f^m(n)=\frac{1}{8^m}\left(9^m n+7\sum_{k=0}^{m-1}9^{m-k}\,8^k\right).\]

Consider
\[f^m(n)=\frac{1}{8^m}\left(9^m n+7\sum_{k=0}^{m-1}9^{m-k}\,8^k\right)\leq 9n.\]

\vspace{12pt}
Solving for $n$, we obtain
\[n\geq\frac{7}{9\cdot 8^m-9^m}\sum_{k=0}^{m-1}9^{m-k}\,8^k.\]

However, $9\cdot 8^m-9^m$ is positive only for $m\leq 18$.  So let $m=18$, then for
\[n\geq 692>\frac{1}{9\cdot 8^{18}-9^{18}}\sum_{k=0}^{17}9^{18-k}\,8^k\]
there are at least eighteen prime numbers between $n$ and $9n$.
\end{proof}
\vspace{10pt}
\section{Between $n$ and $520n$}
\label{sec:BetweenNand520N}

Similarly to the previous two sections, using Theorem 2.3.4, we may establish that there are at least 519 prime numbers between $n$ and $520n$ for all positive integers $n>7$.

\vspace{10pt}
\begin{Theorem}
For $n>7$ there are at least $519$ prime numbers between $n$ and $520n$.
\end{Theorem}

\vspace{-10pt}
\begin{proof}
The cases for $n=8, 9, \ldots, 31408$ may be verified directly.  Now let $\nolinebreak{n\geq 31409}$.  By Theorem 2.3.4 we know there exists primes $p_1$, $p_2$, $\ldots$, and $p_{519}$ such that $\nolinebreak{p_1\in(n, 2n)}$, $p_2\in(2n, 3n)$, $\ldots$, and $p_{519}\in(519n, 520n)$.

Thus $p_1, p_2, \ldots, p_{519}\in(n, 520n)$ as desired.
\end{proof}

\vspace{-10pt}
In the next theorem we improve on the number of prime numbers between $n$ and $520n$ to show that there are at least $3248$ prime numbers between $n$ and $520n$ for $n>58$.

\vspace{10pt}
\begin{Theorem}
For $n>58$ there are at least $3248$ prime numbers between $n$ and $520n$.
\end{Theorem}

\vspace{-10pt}
\begin{proof}
The cases when $n=59, 60, \ldots, 31408$ may be verified directly.  Now let $\nolinebreak{f(n)=\frac{520n}{519}}$ for $n\geq 31409$ and let $f^m(n)=f(f^{m-1}(n))$.  By Theorem 2.3.1 there exists a prime number between $n$ and $f(n)$.  Furthermore, there exists a prime number between $f^{m-1}(n)$ and $f^m(n)$ for all $\nolinebreak{m\in\mathbb{N}\setminus\{1\}}$.  Consider:
\[f^{3248}(n)=\frac{2^{9744}\,5^{3248}\,13^{3248}\,n}{3^{3248}\,173^{3248}}<519.14n<520n.\]

Which holds for all $n>0$ and our proof is complete.
\end{proof}
\section{Between $n$ and $kn$}
\label{sec:BetweenNandKN}

We will now generalize our previous results by showing that for any $n\geq k\geq 2$ there are at least $k-1$ prime numbers in the interval $(n,kn)$.  In order to obtain our result we will use the following inequality as shown by P. Dusart in Theorem 5.2 of \cite{dusart2010}:
\[\mid\vartheta(x)-x\mid<\frac{3.965x}{\log^2 x},\]
where $\log^2 x=(\log x)^2$.

\begin{Theorem}
For all $n\geq k\geq 2$, there are at least $k-1$ prime numbers between $n$ and $kn$.
\end{Theorem}

\vspace{-12pt}
\begin{proof}
The cases for $n\geq k$ and $7\geq k\geq 2$ follow directly by \cite{bachraoui2006}, \cite{loo2011}, Theorem 2.1.3, Corollary 2.2.12, Corollary 2.2.13, Corollary 2.2.14, and Theorem 2.2.10.

Let $n\geq k\geq 8$.  By Theorem 5.2 of \cite[p.~4]{dusart2010}, for $x\geq 2$, we obtain
\[\mid\vartheta(x)-x\mid<\frac{3.965x}{\log^2 x}.\]

Now
\vspace{8pt}
\begin{equation}
\vartheta(kn)-\vartheta(n)>\left(kn-\frac{3.965kn}{\log^2 kn}\right)-\left(n+\frac{3.965n}{\log^2 n}\right).
\end{equation}

\vspace{8pt}
Moreover, observe that the following inequalities are equivalent for all $n\geq k\geq 8$:
\[kn-\frac{3.965kn}{\log^2 kn}-n-\frac{3.965n}{\log^2 n}>(k-1)\log kn,\]
\[k-1>\frac{(k-1)\log kn}{n}+\frac{3.965k}{\log^2 kn}+\frac{3.965}{\log^2 n}.\]

Now the right-hand side is decreasing in $n$, so it suffices to verify the case when $n=k$.  When $n=k$, we obtain the following equivalent inequalities:
\[k-1>\frac{(k-1)\log k^2}{k}+\frac{3.965k}{\log^2 k^2}+\frac{3.965}{\log^2 k},\]
\[1>\frac{2\log k}{k}+\frac{3.965k}{4(k-1)\log^2 k}+\frac{3.965}{(k-1)\log^2 k}.\]

Now the right-hand side is decreasing in $k$, so it suffices to verify the case when $k=8$.  When $k=8$, we obtain
\[1>\frac{2\log 8}{8}+\frac{3.965\cdot8}{28\log^2 8}+\frac{3.965}{7\log^2 8}.\]

Thus 
\[\vartheta(kn)-\vartheta(n)>\left(kn-\frac{3.965kn}{\log^2 kn}\right)-\left(n+\frac{3.965n}{\log^2 n}\right)>(k-1)\log kn\]
for $n\geq k\geq 8$.

That is,
\[\vartheta(kn)-\vartheta(n)=\sum_{n<p\leq kn}\log(p)>(k-1)\log kn.\]

However, $\log kn\geq\log p$ for any $p$ satisfying $n<p\leq kn$ and hence for the above inequality to be true there must be at least $k-1$ prime numbers between $n$ and $kn$.
\end{proof}

Equation $(3.1)$ in the previous theorem provides us with a better lower approximation for $\vartheta(kn)-\vartheta(n)$ as in the following corollary.

\begin{Corollary}
For all $n\geq k\geq 8$,
\[\vartheta(kn)-\vartheta(n)>n\left(k-1-3.965\left(\frac{k}{\log^2 kn}+\frac{1}{\log^2 n}\right)\right).\]
\end{Corollary}
\section{Between Successive Powers}
\label{sec:BetweenPowers}

If we allow $k=n$ in Theorem 3.4.1, then there are at least $n-1$ prime numbers between $n$ and $n^2$ for all $n\geq 64$.  However, we may generalize our method of proof to show that there are at least $n^{d-(\log d)/2}$ many prime numbers between $n^d$ and $n^{d+1}$ as in the next theorem.

\begin{Theorem}
For $n\geq 8$ and $d\geq 1$, there are at least $n^{d-(\log d)/2}$ prime numbers between $n^d$ and $n^{d+1}$.
\end{Theorem}

\vspace{-12pt}
\begin{proof}
By Theorem 5.2 of \cite[p.~4]{dusart2010}, for $x\geq 2$, we obtain
\[\mid\vartheta(x)-x\mid<\frac{3.965x}{\log^2 x}.\]

Now
\[\vartheta(n^{d+1})-\vartheta(n^d)>\left(n^{d+1}-\frac{3.965n^{d+1}}{\log^2 n^{d+1}}\right)-\left(n^d+\frac{3.965n^d}{\log^2 n^d}\right).\]

Moreover, observe that the following inequalities are equivalent for all $n\geq 8$ and $d\geq 1$:

\vspace{-8pt}
\[n^{d+1}>n^{d-(\log d)/2}\log n^{d+1}+n^d+\frac{3.965n^{d+1}}{\log^2 n^{d+1}}+\frac{3.965n^d}{\log^2 n^d},\]
\[1>\frac{\log n^{d+1}}{n^{1+(\log d)/2}}+\frac{1}{n}+\frac{3.965}{\log^2 n^{d+1}}+\frac{3.965}{n\log^2 n^d}.\]

Now the right-hand side is decreasing in both $n$ and $d$, so it suffices to verify the case when $n=8$ and $d=1$.  When $n=8$ and $d=1$, we obtain
\[1>\frac{3\log 2}{4}+\frac{1}{8}+\frac{3.965}{36\log^2 2}+\frac{3.965}{72\log^2 2}.\]

Thus
\[\vartheta(n^{d+1})-\vartheta(n^d)>\left(n^{d+1}-\frac{3.965n^{d+1}}{\log^2 n^{d+1}}\right)-\left(n^d+\frac{3.965n^d}{\log^2 n^d}\right)>n^{d-(\log d)/2}\log n^{d+1}\]
for $n\geq 8$ and $d\geq 1$.

That is,
\[\vartheta(n^{d+1})-\vartheta(n^d)=\sum_{n^d<p\leq n^{d+1}}\log p>n^{d-(\log d)/2}\log n^{d+1}.\]

However, $\log n^{d+1}\geq\log p$ for any $p$ satisfying $n^d<p\leq n^{d+1}$ and hence for the above inequality to be true there must be at least $n^{d-(\log d)/2}$ prime numbers between $n^d$ and $n^{d+1}$.
\end{proof}
\chapter{The Prime Number Theorem\label{chapter4}}
\setcounter{Theorem}{0}
The prime number theorem describes the asymptotic distribution of the prime numbers among the positive integers.  That is to say, the prime number theorem asserts that the prime numbers become rarer as they become larger.  This is formalized by
\[\lim_{x\to\infty}\frac{\pi(x)}{x/\log x}=1\]
which is known as the asymptotic law of distribution of prime numbers.  Equivalently,
\begin{equation}
\pi(x)\sim\frac{x}{\log x}.
\end{equation}

The first truly elementary proofs of the prime number theorem appeared within \cite{erdos1949, selberg1949} and was also cause to the Erd\H{o}s--Selberg priority dispute, see \cite{spencer2009}.

Now by equation $(4.1)$ for $k$ a positive integer we would expect the number of prime numbers in the interval $\nolinebreak{(kn,(k+1)n)}$ to increase as $n$ runs through the positive integers.  Simply applying the prime number theorem estimates that we could expect, very roughly, $\frac{(k+1)n}{\log (k+1)n}-\frac{kn}{\log kn}\sim\frac{n}{\log kn}$ many prime numbers in the interval $\nolinebreak{(kn,(k+1)n)}$.

We will establish explicit lower bounds for the number of prime numbers in each of our intervals in the previous two chapters.  In each case we first show a lower approximation for the number of prime numbers in a particular interval and then show that our results coincide with the prime number theorem.

To establish an explicit bound for the number of prime numbers in the interval $(4n,5n)$ we may bound each prime in the interval from above by $5n$ as follows.

\begin{Theorem}
For $n>2$, the number of prime numbers in the interval $(4n, 5n)$ is at least
\[\log_{5n}\left[\frac{0.054886}{2^{\frac{n}{2}}n^{\frac{3}{2}}}{\left(\frac{3125}{256}\right)}^{\frac{n}{6}}{(5n)}^{-\frac{2.51012\sqrt{5n}}{\log(5n)}}\right].\]
\end{Theorem}
\begin{proof}
In Theorem 2.1.3 we approximated the product of prime numbers between $4n$ and $5n$ from below by
\[\frac{0.054886}{2^{\frac{n}{2}}n^{\frac{3}{2}}}{\left(\frac{3125}{256}\right)}^{\frac{n}{6}}{(5n)}^{-\frac{2.51012\sqrt{5n}}{\log(5n)}}.\]
Bounding each of the prime numbers between $4n$ and $5n$ from above by $5n$, we obtain:
\begin{align*}
&\log_{5n}\left[\frac{0.054886}{2^{\frac{n}{2}}n^{\frac{3}{2}}}{\left(\frac{3125}{256}\right)}^{\frac{n}{6}}{(5n)}^{-\frac{2.51012\sqrt{5n}}{\log(5n)}}\right]\\
&=\frac{\log 0.054886-\frac{n}{2}\log 2+\frac{n}{6}\log \frac{3125}{256}-2.51012\sqrt{5n}+\frac{3}{2}\log 5}{\log n+\log 5}-\frac{3}{2}\\
&>\frac{n\left(\frac{1}{6}\log \frac{3125}{256}-\frac{1}{2}\log 2-\frac{2.51012\sqrt{5}}{\sqrt{n}}\right)}{2\log n}+\frac{\frac{3}{2}\log 5+\log 0.054886}{\log n+\log 5}-\frac{3}{2}\\
&>\frac{n}{\log n}\left(0.035214-\frac{2.8063995}{\sqrt{n}}\right)-\frac{168033}{100000}.
\end{align*}

Now observe that $\displaystyle\lim_{n\to\infty}\frac{2.8063995}{\sqrt{n}}=0$.  Moreover $\displaystyle\lim_{n\to\infty}\frac{n}{\log n}=+\infty$.
\end{proof}

Comparing the previous result with the weak version of the prime number theorem, $\tfrac{(k+1)n}{\log (kn+n)}-\tfrac{kn}{\log kn}$ with $k=4$, we obtain the following table.
\begin{table}[ht]
\centering
\begin{tabular}{l | r | r | r}
$n$ & Result & Weak PNT & $\pi(5n)-\pi(4n)$\\
\hline
$10^4$ & 11.7 & 846.4 & 930\\
$10^5$ & 399.9 & 7093.3 & 7678\\
$10^6$ & 4200.5 & 61023.5 & 65367\\
$10^7$ & 38725.5 & 535330.3 & 567480\\
$10^8$ & 348808.9 & 4767502.3 & 5019541\\
\end{tabular}\\
\parbox{5cm}{\caption{Comparison of Theorem 4.0.1 with PNT.}}
\end{table}

By Theorem 4.0.1 we have the following theorem.

\begin{Theorem}
As $n\rightarrow\infty$, the number of prime numbers in the interval $[4n, 5n]$ goes to infinity.  That is, for every positive integer $m$, there exists a positive integer $L$ such that for all $n\geq L$, there are at least $m$ prime numbers in the interval $[4n, 5n]$.
\end{Theorem}

To establish an explicit lower bound for the number of prime numbers in the interval $(8n, 9n)$ we may bound each prime in the interval from above by $9n$ as in the next theorem.

\begin{Theorem}
For $n>4$, the number of prime numbers in the interval $(8n,9n)$ is at least
\[\log_{9n}\left[32549.3e^{-1.001102\left(\frac{9n}{19}\right)}n^{-13}{\left(\frac{9^9}{8^8}\right)}^{\frac{17n}{105}}{(9n)}^{-\frac{2.51012\sqrt{9n}}{\log(9n)}}\right].\]
\end{Theorem}
\begin{proof}
In Theorem 2.2.10 we approximated the product of prime numbers between $8n$ and $9n$ from below by
\[32549.3e^{-1.001102\left(\frac{9n}{19}\right)}n^{-13}{\left(\frac{9^9}{8^8}\right)}^{\frac{17n}{105}}{(9n)}^{-\frac{2.51012\sqrt{9n}}{\log(9n)}}.\]
Bounding each of the prime numbers between $8n$ and $9n$ from above by $9n$, we obtain:
\begin{align*}
&\log_{9n}\left[32549.3e^{-1.001102\left(\frac{9n}{19}\right)}n^{-13}{\left(\frac{9^9}{8^8}\right)}^{\frac{17n}{105}}{(9n)}^{-\frac{2.51012\sqrt{9n}}{\log(9n)}}\right]\\
&\quad>\frac{n}{2\log n}\left(\frac{17}{105}\log\left(\frac{9^9}{8^8}\right)-1.001102\left(\frac{9}{19}\right)+\frac{\log 32549.3}{n}-\frac{13\log n}{n}-\frac{7.53036}{\sqrt{n}}\right)\\
&\quad>\frac{n}{\log n}\left(0.0170459-\frac{13\log n}{n}-\frac{3.76518}{\sqrt{n}}\right).
\end{align*}
Now observe that $\displaystyle\lim_{n\to\infty}\left(\frac{13\log n}{n}+\frac{3.76518}{\sqrt{n}}\right)=0$.  Moreover $\displaystyle\lim_{n\to\infty}\frac{n}{\log n}=+\infty$.
\end{proof}

Comparing the previous result with the weak version of the prime number theorem, $\tfrac{(k+1)n}{\log (kn+n)}-\tfrac{kn}{\log kn}$ with $k=8$, we obtain the following table.

\begin{table}[ht]
\centering
\begin{tabular}{l | r | r | r}
$n$ & Result & Weak PNT & $\pi(9n)-\pi(8n)$\\
\hline
$10^4$ & -55.8 & 803.4 & 876\\
$10^5$ & 71.9 & 6788.2 & 7323\\
$10^6$ & 1732.4 & 58748.2 & 62712\\
$10^7$ & 17669.9 & 517719.7 & 547572\\
$10^8$ & 163011.7 & 4627221.6 & 4863036\\
\end{tabular}\\
\parbox{5cm}{\caption{Comparison of Theorem 4.0.3 with PNT.}}
\end{table}

By Theorem 4.0.3 we have the following theorem.

\begin{Theorem}
As $n\rightarrow\infty$, the number of prime numbers in the interval $[8n,9n]$ goes to infinity.  That is, for every positive integer $m$, there exists a positive integer $L$ such that for all $n\geq L$, there are at least $m$ prime numbers in the interval $[8n,9n]$.
\end{Theorem}

Similar results may be shown about the number of prime numbers between $n$ and $kn$.  That is, the number of prime numbers between $n$ and $kn$ tends to infinity as $n$ tends to infinity.

\begin{Theorem}
For all $n\geq k\geq 8$, the product of prime numbers between $n$ and $kn$ is at least
\[\exp\left[n\left(k-1-3.965\left(\frac{k}{\log^2 kn}+\frac{1}{\log^2 n}\right)\right)\right].\]
\end{Theorem}
\begin{proof}
By Corollary 3.4.2, we obtain

\vspace{-12pt}
\[\vartheta(kn)-\vartheta(n)>n\left(k-1-3.965\left(\frac{k}{\log^2 kn}+\frac{1}{\log^2 n}\right)\right).\]

Taking both sides of the inequality to the base $e$, we obtain

\vspace{-24pt}
\begin{align*}
e^{\vartheta(kn)-\vartheta(n)}&=\prod_{n<p\leq kn}p\\ &=\prod_{n<p<kn}p\\
&>\exp\left[n\left(k-1-3.965\left(\frac{k}{\log^2 kn}+\frac{1}{\log^2 n}\right)\right)\right]
\end{align*}
as desired.
\end{proof}

\begin{Theorem}
For all $n\geq k\geq 8$, the number of prime numbers in the interval $(n,kn)$ is at least
\vspace{-12pt}
\[\frac{n}{\log kn}\left(k-1-3.965\left(\frac{k}{\log^2 kn}+\frac{1}{\log^2 n}\right)\right).\]
\end{Theorem}
\begin{proof}
Bounding the product of prime numbers between $n$ and $kn$ from above by $\log kn$ as noted in Theorem 4.0.5, we obtain
\begin{multline*}
\log_{kn}\left[\exp\left[n\left(k-1-3.965\left(\frac{k}{\log^2 kn}+\frac{1}{\log^2 n}\right)\right)\right]\right]\\=\frac{n}{\log kn}\left(k-1-3.965\left(\frac{k}{\log^2 kn}+\frac{1}{\log^2 n}\right)\right).\qedhere
\end{multline*}
\end{proof}
\chapter{Summary and Conclusions\label{chapter5}}

In Chapter 2 we showed there exists a prime number in each of the aforementioned intervals.  We also provided an argument as to why the methods of Section 2.1 may not be able to prove any higher cases.

We may supply a similar argument to the methods of Section 2.2.  Also, note that $B_k(n,m)$ is not much larger than $\sqrt{\binom{(k+1)n/m}{kn/m}}$ and applying Stirling's approximation shows this fact.  Perhaps it is possible to instead bound each product of prime numbers composing $T_2$ by $\sqrt{\binom{(k+1)n/m}{kn/m}}$ and extend the method of proof for a few more cases.

Ultimately it appears that the elementary methods become too cumbersome to be effectively utilized.  Certainly this is demonstrated within the results of this thesis.

Of course all of the results concerning prime numbers in intervals within this research are guaranteed eventualities due to the prime number theorem.  In fact, for any $\varepsilon>0$ there is a prime number between $x$ and $(1+\varepsilon)x$ for $x$ sufficiently large, see \cite{erdos1949}.  The true difficulty lies in determining what $x$ is large enough for any given $\varepsilon$.

\singlespacing
\bibliographystyle{plain}

\cleardoublepage
\ifdefined\phantomsection
  \phantomsection  
\else
\fi
\addcontentsline{toc}{chapter}{Bibliography}

\bibliography{biblio}

\begin{thebibliography}{10}

\bibitem{bachraoui2006}
M.~El Bachraoui.
\newblock {Primes in the Interval $[2n, 3n]$}.
\newblock {\em {Int. J. Contemp. Math. Sci.}}, 1:617--621, 2006.

\bibitem{baker2001}
R.~C. Baker, G.~Harman, and J.~Pintz.
\newblock {The difference between consecutive primes, II}.
\newblock {\em {Proc. of the London Math. Soc.}}, 83(3):532--562, 2001.

\bibitem{chebyshev1850}
P.~Chebyshev.
\newblock {M\'{e}moire sur les nombres premiers}.
\newblock {\em {M\'{e}m. Acad. Sci. St. Pétersbourg}}, 7:17--33, 1850.

\bibitem{chen1975}
J.~R. Chen.
\newblock {On the Distribution of Almost Primes in an Interval}.
\newblock {\em {Sci. Sinica}}, 18:611--627, 1975.

\bibitem{dusart2010}
P.~Dusart.
\newblock {Estimates of Some Functions Over Primes Without R. H.}
\newblock http://arxiv.org/abs/1002.0442. arXiv:1002.0442v1 [math.NT], 2010.

\bibitem{erdos1932}
P.~Erd\H{o}s.
\newblock {Beweis eines satzes von tschebyschef}.
\newblock {\em {Acta Litt. Univ. Sci., Szeged, Sect. Math.}}, 5:194--198, 1932.

\bibitem{erdos1949}
P.~Erd\H{o}s.
\newblock {On a new method in elementary number theory which leads to an
  elementary proof of the prime number theorem}.
\newblock {\em {Proc. Nat. Acad. Scis.}}, 35:374--384, 1949.

\bibitem{erdos2003}
P.~Erd\H{o}s and J.~Sur{\'a}nyi.
\newblock {\em Topics in the Theory of Numbers}.
\newblock Springer Verlag, 2003.

\bibitem{hardy1979}
G.~H. Hardy and E.~M. Wright.
\newblock {\em An Introduction to the Theory of Numbers, 5th ed.}
\newblock Oxford, England: Oxford University Press, 1979.

\bibitem{iwaniec1984}
H.~Iwaniec and J.~Pintz.
\newblock {Primes in Short Intervals}.
\newblock {\em {Monatsh. f. Math}}, 98:115--143, 1984.

\bibitem{loo2011}
A.~Loo.
\newblock {On the Primes in the Interval $[3n, 4n]$}.
\newblock {\em {Int. J. Contemp. Math. Sci.}}, 6:1871--1882, 2011.

\bibitem{nagura1952}
J.~Nagura.
\newblock {On The Interval Containing At Least One Prime Number}.
\newblock {\em {Proceedings of the Japan Academy, Series A}}, 28:177--181,
  1952.

\bibitem{ramanujan1919}
S.~Ramanujan.
\newblock {A Proof of Bertrand's Postulate}.
\newblock {\em {J. of the Indian Math. Soc.}}, 11:181--182, 1919.

\bibitem{robbins1955}
H.~Robbins.
\newblock {A Remark on Stirling's Formula}.
\newblock {\em {Amer. Math. Monthly}}, 62:26--29, 1955.

\bibitem{rosser1962}
J.~Rosser and L.~Schoenfeld.
\newblock {Approximate formulas for some functions of prime numbers}.
\newblock {\em {Illinois J. Math.}}, 6:64--94, 1962.

\bibitem{schoenfeld1976}
L.~Schoenfeld.
\newblock {Sharper bounds for the Chebyshev functions $\theta(x)$ and
  $\psi(x)$. II}.
\newblock {\em {Math. Comp.}}, 30:337--360, 1976.

\bibitem{selberg1949}
A.~Selberg.
\newblock {An elementary proof of the prime-number theorem}.
\newblock {\em {Ann. of Math.}}, 50(2):305--313, 1949.

\bibitem{spencer2009}
J.~Spencer and R.~L. Graham.
\newblock {The Elementary Proof of the Prime Number Theorem}.
\newblock {\em {The Math. Intelligencer}}, 31(3):18--23, 2009.

\end{thebibliography}

\end{document}